\documentclass[12pt]{article}

\usepackage[numbers]{natbib}
\usepackage{amssymb, amsmath, url, amsthm, amssymb}
\usepackage[T1]{fontenc}
\usepackage[latin1]{inputenc}
\usepackage{color}
\usepackage{bbm}
\usepackage{graphicx}
\usepackage{multirow}
\usepackage{rotating}
\usepackage{enumerate}

\newtheorem{Theorem}{Theorem}

\newtheorem{Corollary}{Corollary}
\newtheorem{Remark}{Remark}
\newtheorem{Proposition}{Proposition}

\newtheorem{Definition}{Definition}

\newenvironment{assumption}[1][Assumption]{\begin{trivlist}
\item[\hskip \labelsep {\bfseries #1}]}{\end{trivlist}}

\newcommand{\convP}{\stackrel{\mathcal{P}}{\longrightarrow}}

\newcommand{\ii}{\mathbbm{i}}
\newcommand{\tr}{\mathrm{tr}}
\newcommand{\spa}{\overline{\mathrm{sp}}}
\newcommand{\bv}{\mathbf{v}}
\newcommand{\bx}{\mathbf{x}}

\newcommand{\bX}{\mathbf{X}}
\newcommand{\cS}{\mathcal{S}}

\newcommand{\bphi}{\boldsymbol{\phi}}

\newcommand{\SPDO}{\mathcal{F}}

\newcommand{\ip}[2]{\langle #1,#2 \rangle}

\newtheorem{lemma}{Lemma}

\newcommand{\Sum}{\mathop{\sum}\limits}

\DeclareMathOperator{\cov}{cov} \DeclareMathOperator{\Var}{Var}

\begin{document}

\title{Dynamic Functional Principal Components
}
\author{Siegfried H\"ormann$^1$\thanks{Corresponding author. Email: shormann@ulb.ac.be} \and {\L}ukasz Kidzi\'nski$^{1}$ \and Marc Hallin$^{2,3}$} 
\date{}
\maketitle
$^{1}$ {\small Department of Mathematics, Universit\'e libre de Bruxelles (ULB), CP210, Bd.\ du Triomphe, B-1050 Brussels, Belgium.}

$^2$ {\small ECARES,
Universit\'e libre de Bruxelles (ULB), CP 114/04
50, avenue F.D. Roosevelt
B-1050 Brussels,
Belgium.}

$^3$ {\small ORFE,
Princeton University,
Sherrerd Hall,
Princeton,
NJ 08540,
USA.}

\begin{abstract}
{\bf Abstract.} In this paper, we address the problem of dimension reduction for time series of functional data $(X_t\colon t\in\mathbb{Z})$. Such {\it functional time series}  frequently arise, e.g., when a continuous-time process is segmented into some smaller natural units, such as days. Then each~$X_t$ represents one intraday curve. We argue that functional principal component analysis (FPCA), though a key technique in the field and a benchmark for any competitor, does not provide an adequate dimension reduction in a time-series setting. FPCA indeed is a {\it static} procedure which ignores the essential information provided by   the serial dependence structure of the functional data under study. Therefore, inspired by Brillinger's theory of {\it dynamic principal components}, we propose a {\it dynamic}  version of FPCA, which is based on a frequency-domain approach. By means of a simulation study and an empirical illustration, we show the considerable improvement the dynamic approach  entails when compared to the usual static procedure.\vspace{-.5mm}
\end{abstract}

\noindent \textbf{Keywords.} Dimension reduction, frequency domain analysis, functional data analysis, functional time series, functional spectral analysis, principal components, Kar\-hu\-nen-Lo\`eve expansion.

\section{Introduction}\label{se:intro}

The tremendous technical improvements in data collection and storage allow to get an increasingly complete picture of many common phenomena. In principle, most processes in real life are continuous in time and, with improved data acquisition techniques, they can be recorded at arbitrarily high frequency. To benefit from increasing information, we need appropriate statistical tools that can help extracting the most important characteristics of some possibly high-dimensional specifications. Functional data analysis (FDA), in recent years, has proven to be an appropriate tool in many such cases and has consequently evolved into a very important field of research in the statistical community. 

Typically, functional data are considered as realizations of (smooth) random curves. Then every observation $X$ is a curve $(X(u)\colon u\in\mathcal{U})$. One generally assumes, for simplicity, that $\mathcal{U}=[0,1]$, but $\mathcal{U}$ could be a more complex domain like a cube or the surface of a sphere. Since observations are functions, we are dealing with high-dimensional -- in fact intrinsically infinite-dimensional -- objects. So, not surprisingly, there is a demand for efficient data-reduction techniques. As such, {\em functional principal component analysis} (FPCA) has taken a leading role in FDA, and  {\em functional principal components} (FPC)  arguably  can be seen as {\it the} key technique in the field. 

In analogy to classical multivariate PCA (see Jolliffe~\cite{jolliffe:2002}), functional PCA  relies on an eigendecomposition of the underlying covariance function. The mathematical foundations for this have been laid several decades ago in the pioneering papers by Karhunen~\cite{karhunen:1947} and Lo\`eve~\cite{loeve:1946}, but it took a while until the method was popularized in the statistical community. Some earlier contributions are Besse and Ramsay~\cite{besse:ramsay:1986}, Ramsay and Dalzell~\cite{ramsay:dalzell:1991} and, later, the influential books by Ramsay and Silverman~\cite{ramsay:silverman:2002},~\cite{ramsay:silverman:2005} and Ferraty and Vieu~\cite{ferraty:vieu:2006}. Statisticians have been working on problems related to estimation and inference (Kneip and Utikal~\cite{kneip:utikal:2001}, Benko et al.~\cite{benko:hardle:kneip:2009}), asymptotics (Dauxois et al.~\cite{dauxois:pousse:romain:1982} and Hall and Hosseini-Nasab~\cite{hall:hosseini:2006}), smoothing techniques (Silverman~\cite{silverman:1996}), sparse data (James et al.~\cite{james:hastie:sugar:2000}, Hall et al.~\cite{hall:muller:wang:2006}), and robustness issues (Locantore et al.~\cite{locantore:etal:1999}, Gervini~\cite{gervini:2007}), to name just a few. Important applications include FPC-based estimation of functional linear models (Cardot et al.~\cite{cardot:ferraty:sarda:1999}, Reiss and Ogden~\cite{reiss:ogden:2007}) or forecasting (Hyndman and Ullah~\cite{hyndman:ullah:2007}, Aue et al.~\cite{aue:dubart:hoermann:2012}). The usefulness of functional PCA has also been recognized in other scientific disciplines, like chemical engineering (Gokulakrishnan et al.~\cite{gokulakrishnan:etal:2006}) or functional magnetic resonance imaging (Aston and Kirch~\cite{aston:kirch:2012}, Viviani et al.~\cite{viviani:etal:2005}).  Many more references can be found in the above cited papers and in Sections~8--10 of Ramsay and Silverman~\cite{ramsay:silverman:2005}, where we refer to for background reading. 

Most existing concepts and methods in FDA, even though they may tolerate some amount of serial dependence, have been developed for independent observations. This is a serious weakness, as in numerous applications the functional data under study are obviously dependent, either in time or in space. Examples include daily
curves of financial transactions, daily patterns of geophysical and environmental data, annual temperatures measured on the surface of the earth,  etc. In such cases, we should view the data as the realization of a {\em functional time series} $(X_t(u)\colon t\in\mathbb{Z})$, where the time parameter $t$ is discrete and the parameter $u$ is continuous. For example, in case of daily observations, the curve $X_t(u)$ may be viewed as the observation on day $t$ with intraday time parameter $u$. A key reference on functional time series techniques is Bosq~\cite{bosq:2000}, who studied functional versions of AR processes. We also refer to H\"ormann and Kokoszka~\cite{hormann:kokoszka:2012} for a survey.  

Ignoring serial dependence in this time-series context may result in misleading conclusions and inefficient procedures. H\"ormann and Kokoszka \cite{hormann:kokoszka:2010} investigate the robustness properties of some classical FDA methods in the presence of serial dependence. Among others, they show that usual FPCs still can be consistently estimated within a quite general dependence framework. Then the basic problem, however, is not about consistently estimating traditional FPCs: the problem is that, in a time-series context, traditional FPCs are not the adequate concept of dimension reduction  anymore -- a fact which, since the seminal work of Brillinger~\cite{brillinger:2001}, is well recognized in the usual vector time-series setting. FPCA indeed operates in a {\it static} way: when applied to serially dependent curves, it fails to take into account the potentially very valuable information carried by the past values of the functional observations under study.
In particular, a static FPC with small eigenvalue, hence negligible instantaneous impact on $X_t$, may have a major impact on $X_{t+1}$, and high predictive value. 

Besides their failure to produce optimal dimension reduction, static FPCs, while cross-sectionally uncorrelated at fixed time $t$, typically still exhibit lagged cross-correlations. Therefore the resulting FPC scores cannot be analyzed componentwise as in the i.i.d.\ case, but need to be considered as vector time series which are less easy to handle and interpret.

These major shortcomings are motivating the present development of {\em dynamic functional principal components} (dynamic FPCs). The idea is to transform the functional time series into a vector time series (of low dimension, $\leq 4$, say), where the individual component processes are mutually uncorrelated (at all leads and lags; {\it auto}correlation is allowed, though), and account for most of the dynamics and variability of the original process. The analysis of the functional time series can then be performed on those dynamic FPCs; thanks to  their mutual orthogonality,  
  dynamic FPCs  moreover can be    analyzed componentwise. In analogy to  static FPCA, the curves can be optimally reconstructed/approximated from the low-dimensional  dynamic FPCs  via a dynamic version of the celebrated {\em Karhunen-Lo\`eve expansion}. 

Dynamic principal components first have been suggested by Brillinger~\cite{brillinger:2001} for vector time series. The purpose of this article is to develop and study a similar approach in a functional setup. The methodology relies on a frequency-domain analysis for functional data, a topic which is still in its infancy (see, for instance, Panaretos and Tavakoli~2013a).

The rest of the paper is organized as follows. In Section~\ref{se:illustration} we give a first illustration of the procedure and sketch two typical applications. In Section~\ref{se:metho}, we describe our approach and state a number of relevant propositions. We also provide some asymptotic features. In Section~\ref{se:pract}, we discuss its computational implementation. After an illustration of the methodology by a real data example on pollution curves in Section~\ref{se:reald}, we evaluate our approach in a simulation study (Section~\ref{se:simul}). Appendices~A and B detail the  mathematical framework and contain the proofs. Some of the more technical results and proofs are outsourced to Appendix~C. 

After the present paper (which has been available on Arxiv since October 2012) was submitted, another paper by Panaretos and Tavakoli (2013b) was published, where similar ideas are proposed. While both papers aim at the same objective of a functional extension of Brillinger's concept, there are essential differences between the solutions developed. The main result in Panaretos and Tavakoli (2013b) is  the existence of a functional process $(X_t^*)$ of rank $q$ which serves as an ``optimal approximation''   to the  process $(X_t)$ under study. The construction of $(X_t^*)$, which is mathematically quite elegant, is based on stochastic integration with respect to some orthogonal-increment (functional) stochastic process $(Z_\omega)$. The disadvantage, from  a statistical perspective, is that this construction is not explicit, and that no finite-sample version of the concept is provided -- only the limiting behavior of the empirical spectral density operator and its eigenfunctions is obtained. Quite on the contrary, our Theorem~\ref{th:dfpc:consistency} establishes the consistency of an empirical, explicitly constructed and easily implementable version of the dynamic scores -- which is what a statistician will be interested in. We also remark that we are working under milder technical conditions.

\section{Illustration of the method}\label{se:illustration}

An impression of how well the proposed method works can be obtained from Figure~\ref{fig:comp}. Its left panel shows ten consecutive intraday curves of some pollutant level (a detailed description of the underlying data is given in Section~\ref{se:reald}). The two panels to the right show one-dimensional reconstructions of these curves. We used static FPCA in the central panel and dynamic FPCA in the right panel.
\begin{figure}[ht]
\centering
\includegraphics[width=12.5cm]{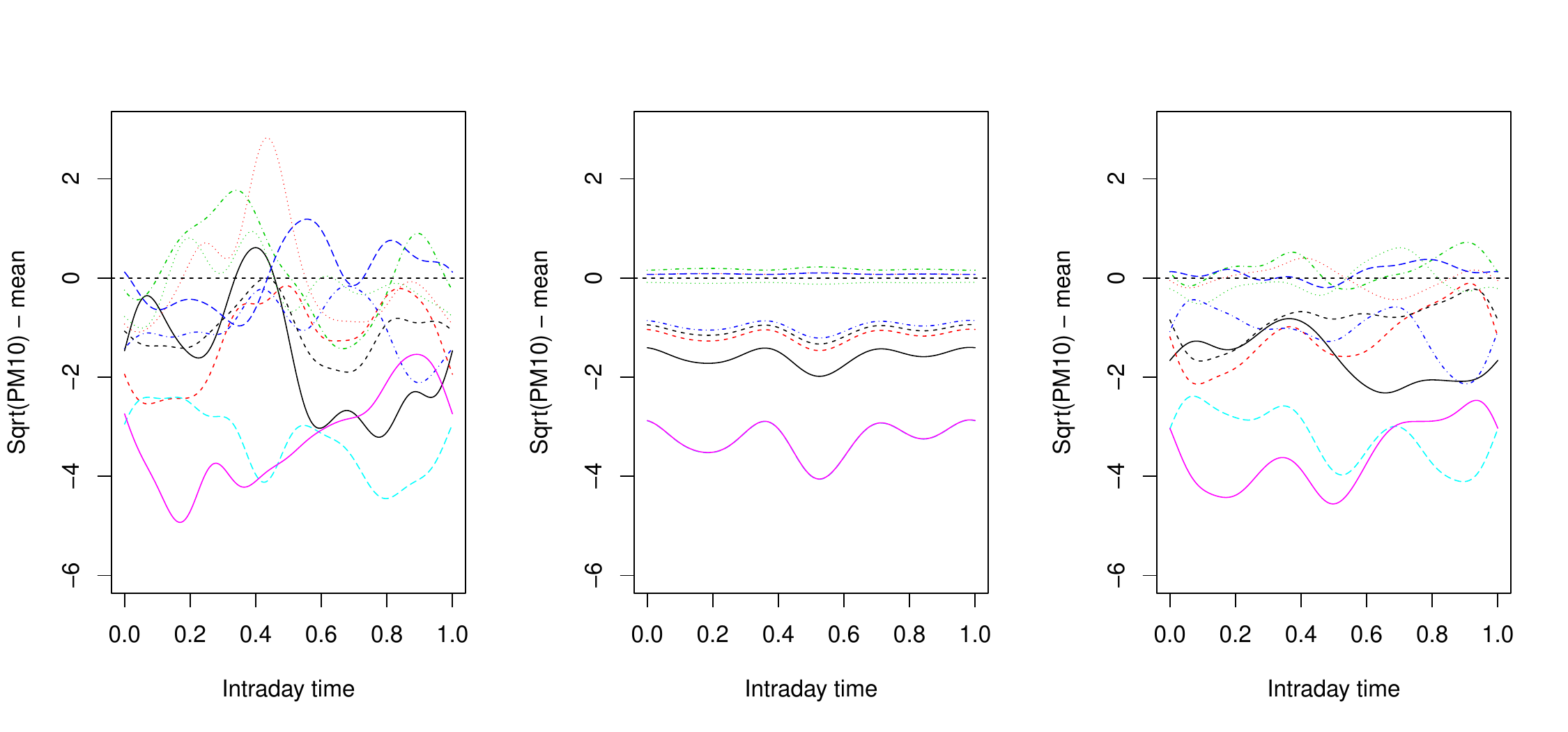}
\caption{\small Ten successive daily observations (left panel), the corresponding  {\it static} Karhunen-Lo\`eve expansion based on  one (static) principal component (middle panel), and the {\it dynamic} Karhunen-Lo\`eve expansion with one dynamic component (right panel). Colors provide the matching between the actual observations and their Karhunen-Lo\`eve approximations. 
}
\label{fig:comp}
\end{figure}
The difference is notable. The static method merely provides an average level,  exhibiting a completely spurious and highly misleading  intraday symmetry. In addition to daily average levels, the dynamic approximation, to a large extent, also catches the intraday evolution of the curves. In particular, it retrieves  the intraday trend of pollution levels, and the location of their daily spikes and troughs (which varies considerably from one curve to the other). For this illustrative example we chose one-dimensional reconstructions, based on one single FPC; needless to say,   increasing the number of FPCs (several principal components), we obtain much better approximations -- see Section~4 for details. 

Applications of dynamic PCA in a time series analysis are the same as those  of static PCA in the context of independent (or uncorrelated) observations. This is why obtaining mutually orthogonal principal components -- in the sense of mutually orthogonal {\it processes} -- is a major issue here. 
This orthogonality, at all leads and lags, of  dynamic principal components, indeed, implies that any second-order based method  (which is the most common approach in time series) can be carried out componentwise, i.e.\  via scalar methods. In contrast,  static  principal components  still have to be treated as a  multivariate time series. 

Let us illustrate this superiority of mutually orthogonal dynamic components over the auto- and cross-correlated static ones by means of two examples.\bigskip

\noindent
\emph{Change point analysis:} Suppose that we wish to find a structural break (change point) in a sequence of functional observations $X_1,\ldots,X_n$. For example, Berkes et al.~\cite{berkes:gabrys:horvath:kokoszka:2009} consider the problem of detecting a change in the mean function of a sequence of independent functional data. They propose to first project data on the~$p$ leading principal components and argue that a change in the mean will show in the score vectors, provided hat the proportion of variance they are accounting for is large enough.  Then a CUSUM procedure is utilized. The test statistic is based on the functional
$$
T_n(x)=\frac{1}{n}\sum_{m=1}^p\hat\lambda_m^{-1}\left(\sum_{1\leq k\leq nx}\hat{Y}^\text{stat}_{mk}-x\sum_{1\leq k\leq n}\hat{Y}^\text{stat}_{mk}\right)^2,\quad 0\leq x\leq 1.
$$
Here $\hat{Y}^\text{stat}_{mk}$ is the $m$-th empirical PC score of $X_k$ and $\hat\lambda_m$ is the $m$-th largest eigenvalue of the empirical covariance operator related to the functional sample.
The assumption of independence  implies that $T_n(x)$ converges, under the no-change hypothesis, to the sum of $p$ squared independent Brownian bridges. Roughly speaking, this is due to the fact that the partial sums of score vectors (used in the CUSUM statistic) converge in distribution  to a multivariate normal  with diagonal covariance. 
That is,  the partial sums of the individual scores become asymptotically  independent, and we just  obtain $p$ independent CUSUM test statistics -- a separate one for each score sequence. The independent test statistics are then aggregated.

This simple structure is lost when data are serially dependent. Then, if a CLT holds,
$
\big(\sum_{1\leq k\leq n}\hat{Y}^\text{stat}_{mk}\colon m=1,\ldots,p\big)^\prime
$
converges to a normal vector where the covariance (which is still diagonal) needs to be replaced by the long-run covariance of the score vectors, which is typically non-diagonal. 

In contrast, using dynamic principal components, the long-run covariance of the score vectors remains diagonal; see Proposition~\ref{pr:secondorder}. Let $\mathrm{diag}(\hat\lambda_1(0),\ldots,\hat\lambda_p(0))$ be a consistent estimator of this long-run variance and $\hat Y_{mk}^\mathrm{dyn}$ be the dynamic scores. Then replacing the test functionals $T_n(x)$ by 
$$
T^\text{dyn}_n(x)=\frac{2\pi}{n}\sum_{m=1}^p\hat\lambda_m^{-1}(0)\left(\sum_{1\leq k\leq nx}\hat{Y}^\text{dyn}_{mk}-x\sum_{1\leq k\leq n}\hat{Y}^\text{dyn}_{mk}\right)^2,\quad 0\leq x\leq 1,
$$
we get that (under appropriate technical assumptions ensuring a functional CLT)  
the same asymptotic behavior holds as for $T_n(x)$,  so that  again $p$ independent CUSUM test statistics can be aggregated.

Dynamic principal components, thus, and not the static ones, provide a feasible extension of the Berkes et al.~\cite{berkes:gabrys:horvath:kokoszka:2009} method to the time series context. \bigskip

\noindent
\emph{Lagged regression:} A lagged regression model is a linear model in which the response $W_t\in\mathbb{R}^q$, say, is allowed to depend on an unspecified number of lagged values of a series of regressor variables $(X_t)\in\mathbb{R}^p$. More specifically, the  model equation is 
\begin{equation}\label{eq:lagged_reg}
W_t=a+\sum_{k\in\mathbb{Z}}b_kX_{t-k}+\varepsilon_t,
\end{equation}
with some i.i.d.\ noise $(\varepsilon_t)$ which is independent of the regressor series. The intercept $a\in\mathbb{R}^q$ and the matrices $b_k\in\mathbb{R}^{q\times p}$ are unknown. In time series analysis, the lagged regression is the natural extension of the traditional linear model for independent data.

The main  problem in this context, which can be tackled by a frequency domain approach, is estimation of the parameters. See, for example, Shumway and Stoffer~\cite{shumway:stoffer:2006} for an introduction. Once the parameters are known, the model can, e.g., be used for prediction.

Suppose now that $W_t$ is a scalar response and that $(X_k)$ constitutes a functional time series.  The corresponding lagged regression model can be formulated in analogy, but involves estimation of an unspecified number of operators, which is quite delicate. 
A pragmatic way to proceed is to have $X_k$ in \eqref{eq:lagged_reg} replaced by the vector of  the first $p$ dynamic functional principal component scores $Y_k=(Y_{1k},\ldots,Y_{pk})^\prime$, say. The general theory implies that, under mild assumptions (basically guaranteeing convergence of the involved series),  
$$
b_k=\frac{1}{2\pi}\int_{-\pi}^\pi B_\theta e^{\ii k\theta}d\theta,
\quad\text{where}\quad
B_\theta=\mathcal{F}_\theta^{WY}\big(\mathcal{F}_\theta^Y\big)^{-1},
$$
and 
$$
\mathcal{F}_\theta^Y=\frac{1}{2\pi}\sum_{h\in\mathbb{Z}}\mathrm{cov}(Y_{t+h},Y_t)e^{-\ii h\theta}\quad\text{and}\quad
\mathcal{F}_\theta^{WY}=\frac{1}{2\pi}\sum_{h\in\mathbb{Z}}\mathrm{cov}(W_{t+h},Y_t)e^{-\ii h\theta}
$$ 
are the spectral density matrix of the score sequence and the cross-spectrum between $(W_t)$ and $(Y_t)$, respectively. In the present setting the structure greatly simplifies. 
Our theory will reveal (see Proposition~\ref{pr:sd_filter0}) that $\mathcal{F}_\theta^Y$ is diagonal at all frequencies and that
$$
B_\theta=\left(\frac{f_\theta^{WY_1}}{\lambda_1(\theta)},\ldots, \frac{f_\theta^{WY_p}}{\lambda_p(\theta)}\right),
$$
with $f_\theta^{WY_m}$ being the co-spectrum between $(W_t)$ and $(Y_{mt})$ and $\lambda_m(\theta)$ is the $m$-th \emph{dynamic eigenvalue} of the spectral density operator of the series $(X_k)$ (see Section~\ref{se:sdoi}). As a consequence, the influence of each score sequence on the regressors can be assessed individually. 

Of course, in applications, these population quantities are replaced by their empirical versions and one may use some testing procedure for the null-hypothesis $H_0\colon
f_\theta^{WY_p}=0$
for all $\theta$,
in order to justify the choice of the dimension of the dynamic score vectors and to retain only those components which have a significant impact on $W_t$.

\section{Methodology for $L^2$ curves}\label{se:metho}
In this section, we introduce some necessary notation and tools. Most of the discussion on technical details is postponed to the Appendices~\ref{app:pract}, \ref{app:largesampleproperties} and \ref{app:technicalresults}. For simplicity,  we are  focusing here on $L^2([0,1])$-valued processes, i.e.\  on square-integrable functions defined on the unit interval;   in the appendices,  however, the theory is developed within a more general framework. 

\subsection{Notation and setup}\label{se:notat}
Throughout this section, we consider a functional time series $(X_t\colon t\in \mathbb{Z})$, where~$X_t$ takes values in the space
$H:=L^2([0,1])$ of complex-valued square-integrable functions on $[0,1]$. This means that $X_t=(X_{t}(u)\colon u\in[0,1])$, with 
$$
\int_0^1|X_t(u)|^2du<\infty\vspace{-2mm}
$$
($|z|:=\sqrt{z\bar z}$, where $\bar z$ the complex conjugate of $z$, stands for the modulus of $z \in \mathbb{C}$).
In most applications, observations are real, but, since we will use spectral methods, a complex vector space definition will serve useful. 

The space $H$ then is a Hilbert space, equipped with the inner product $\langle x,y\rangle:=\int_0^1x(u)\bar{y}(u)du$, so that $\|x\| := {\langle x,x\rangle}^{1/2}$ defines a norm.  The notation $X\in L_H^p$ is used to indicate that, for some $p>0$, $E[\|X\|^p]<\infty$. Any $X\in L_H^1$ then possesses a mean curve $\mu=(E[X(u)]\colon u\in[0,1])$, and any $X\in L_H^2$ a covariance operator $C$, defined by $C(x):=E[(X-\mu)\langle x,X-\mu\rangle]$. 
The operator $C$ is a kernel operator given by
\[
C(x)(u)=\int_0^1c(u,v)x(v)dv,\\\ \text{with}\ c(u,v):=\cov(X(u),X(v)),\,\, u,v\in[0,1],\vspace{-2mm}
\]
with 
 $\cov(X,Y):=E(X-EX)\overline{(Y-EY)}.$
The process $(X_t\colon t\in\mathbb{Z})$ is called {\em weakly stationary} if, for all $t$,    (i)~$X_t\in L_H^2$, (ii) $EX_t=EX_0$, and (iii) for all $h\in \mathbb{Z}$ and~$u,v\in[0,1]$, $$\cov(X_{t+h}(u),X_t(v))=\cov(X_{h}(u),X_0(v))=:c_h(u,v).$$ Denote by $C_h, h\in \mathbb{Z},$ the operator corresponding to the autocovariance kernels $c_h$. Clearly, $C_0=C$. It is well known that, under quite general dependence assumptions, the mean of a stationary functional sequence can be consistently estimated by the sample mean, with the usual $\sqrt{n}$-convergence rate. Since, for our problem, the mean is not really relevant, we throughout suppose that the data have been 
 centered in some preprocessing step. {\em For the rest of the paper, it is tacitly assumed that~$(X_t\colon t\in\mathbb{Z})$ is a weakly stationary, zero mean process defined on some probability space $(\Omega,\mathcal{A},P)$.}

As in the multivariate case, the covariance operator $C$ of a random element $X\in L_H^2$ admits an eigendecomposition (see, e.g., p.\ 178, Theorem 5.1 in \cite{gohberg:goldberg:kaashoek:2003})\vspace{-2mm}
\begin{equation}\label{eq:ed}
C(x)=\sum_{\ell=1}^\infty\lambda_\ell\langle x,v_\ell\rangle v_\ell,\vspace{-2mm}
\end{equation}
where $(\lambda_\ell\colon\ell \geq 1)$ are $C$'s eigenvalues (in descending order) and $(v_\ell\colon\ell \geq 1)$ the corresponding normalized eigenfunctions, so that $C(v_\ell)=\lambda_\ell v_\ell$ and $\|v_\ell\|=1$. 
If $C$ has full rank, then the sequence $(v_\ell\colon\ell \geq 1)$ forms an orthonormal basis 
 of $L^2([0,1])$. Hence, $X$ admits the representation \vspace{-2mm}
 \begin{equation}\label{eq:KLexpansion}
X=\sum_{\ell=1}^\infty\langle X,v_\ell\rangle v_\ell,\vspace{-2mm}
\end{equation} 
which is called the {\em static  Karhunen-Lo\`eve} expansion of $X$.
The eigenfunctions~$v_\ell$ are called {\em the (static) functional principal components (FPCs)} and the coefficients~$\langle X,v_\ell\rangle$ are called {\em the (static) FPC scores} or {\em loadings}. It is well known that the basis~$(v_\ell\colon \ell\geq~1)$ is optimal in representing $X$ in the following sense: if $(w_\ell\colon\ell\geq 1)$ is any other orthonormal basis of $H$, then 
\begin{equation}\label{eq:fpcopt}
E\|X-\sum_{\ell=1}^p \langle X,v_\ell\rangle v_\ell\|^2\leq
E\|X-\sum_{\ell=1}^p \langle X,w_\ell\rangle w_\ell\|^2,\quad \forall p\geq 1.\vspace{-2mm}
\end{equation}
Property \eqref{eq:fpcopt} shows that a finite number of FPCs can be used to approximate the function $X$ by a vector of given dimension $p$ with a minimum loss of ``instantaneous'' information.
It should be stressed, though, that this approximation is of a  {\em static}  nature, meaning that it is performed observation by observation, and does not take into account the possible serial dependence of the $X_t$'s, which is likely to exist in a time-series context. Globally speaking, we should be looking for an approximation which also involves lagged observations, and is based on the whole family $(C_h\colon h\in\mathbb{Z})$ rather than on $C_0$ only. To achieve this goal, we introduce below the {\em spectral density operator}, which contains the full information on the family of operators $(C_h\colon h\in\mathbb{Z})$.

\subsection{The spectral density operator}\label{se:sdoi}
In analogy to the classical concept of a spectral density matrix, we define the {s\it pectral density operator}.

\begin{Definition}\label{def:spectral}
Let $(X_t)$ be a stationary process. The operator $\mathcal{F}_\theta^X$ whose kernel is $$
f^X_\theta(u,v):=\frac{1}{2\pi}\sum_{h\in \mathbb{Z}}c_h(u,v)e^{-\ii h\theta},\quad\theta\in[-\pi,\pi],
\vspace{-2mm}$$
where $\ii$ denotes the imaginary unit, is called the {\em spectral density operator of $(X_t)$ at frequency~$\theta$}.
\end{Definition}

To ensure convergence (in an appropriate sense) of the series defining $f^X_\theta(u,v)$ (see Appendix~\ref{se:sdo}), we impose the following  summability condition on the autocovariances
\begin{equation}\label{e:wd}
\sum_{h\in\mathbb{Z}}\left(\int_0^1\int_0^1|c_h(u,v)|^2 du dv\right)^{1/2}<\infty.
\end{equation}
The same condition is more conveniently expressed as
\begin{equation}\label{e:abssymcov}
\sum_{h\in\mathbb{Z}}\|C_h\|_\mathcal{S}<\infty,\vspace{-2mm}
\end{equation}
where $\|\cdot\|_\mathcal{S}$ denotes the Hilbert-Schmidt norm (see Appendix~\ref{se:lo}). A simple sufficient condition for \eqref{e:abssymcov} to hold will be provided in Proposition~\ref{pr:summab}.

This concept of a spectral density operator has been introduced by Panaretos and Tavakoli~\cite{panaretos:tavakoli:2012}. In our context, this operator is used to create particular {\em functional filters} (see Sections~\ref{ss:dynam} and \ref{ss:ffilters}),  which are the building blocks for the construction of {\em dynamic} FPCs. A functional filter is defined via a sequence $\boldsymbol{\Phi}=(\Phi_\ell\colon \ell\in\mathbb{Z})$ of linear operators between the spaces $H=L^2([0,1])$ and $H^{\prime}=\mathbb{R}^p$. The filtered variables $Y_t$ have the form $Y_t=\sum_{\ell\in\mathbb{Z}}\Phi_\ell(X_{t-\ell})$, and by the Riesz representation theorem, the linear operators $\Phi_\ell$ are given as 
$$x\mapsto\Phi_\ell(x)=(\langle x,\phi_{1\ell}\rangle,\ldots,\langle x,\phi_{p\ell}\rangle)^\prime,\quad\text{with}\quad \phi_{1\ell},\ldots,\phi_{p\ell}\in H.$$
We shall considerer filters $\boldsymbol{\Phi}$ for which the sequences
$
(\sum_{\ell=-N}^N\phi_{m\ell}(u)e^{\ii \ell\theta}\colon N\geq1)
$,
$
1\leq m\leq p,
$
converge in 
$L^2([0,1]\times[-\pi,\pi])$. Hence, we assume existence of a square integrable function $\phi_m^\star(u|\theta)$ such  that
\begin{equation}\label{e:norm_2}
\lim_{N\to\infty}\int_{-\pi}^\pi\int_0^1\left(\sum_{\ell=-N}^N\phi_{m\ell}(u)e^{\ii \ell\theta}-\phi_m^\star(u|\theta)\right)^2dud\theta=0.
\end{equation}
In addition we suppose that 
\begin{equation}\label{e:norm_3}
\sup_{\theta\in[-\pi,\pi]}\int_0^1\left[\phi_m^\star(u|\theta)\right]^2du<\infty.
\end{equation}
Then, we write
$\phi_m^\star(\theta):=\sum_{\ell\in\mathbb{Z}}\phi_{m\ell}e^{\ii \ell\theta}$ or, in order  to emphasize its  functional nature,~$\phi_m^\star(u|\theta):=\sum_{\ell\in\mathbb{Z}}\phi_{m\ell}(u)e^{\ii \ell\theta}$. We denote by $\mathcal{C}$ the family of filters $\boldsymbol{\Phi}$ which satisfy \eqref{e:norm_2} and \eqref{e:norm_3}.
For example, if $\boldsymbol{\Phi}$ is such that $\sum_{\ell}\|\phi_{m\ell}\|<\infty$, then   $\boldsymbol{\Phi}\in \mathcal{C}$.

 The following proposition relates the spectral density operator of $(X_t)$ to the spectral density matrix of the filtered sequence $(Y_t=\sum_{\ell\in\mathbb{Z}}\Phi_\ell(X_{t-\ell}))$. This simple result plays a crucial role in our construction. 

\begin{Proposition}\label{pr:sd_filter} Assume that $\boldsymbol{\Phi}\in\mathcal{C}$ and let $\phi_m^\star(\theta)$ be given as above. Then the series $\sum_{\ell\in\mathbb{Z}}\Phi_\ell(X_{t-\ell})$ converges in mean square to a limit $Y_t$. The $p$-dimensional vector process $(Y_t)$ is stationary, with   spectral density matrix 
$$
\mathcal{F}^Y_\theta=
\begin{pmatrix}
\langle \mathcal{F}^X_\theta(\phi_1^\star(\theta)),\phi_{1}^\star(\theta)\big\rangle &\cdots&\langle \mathcal{F}^X_\theta(\phi_p^\star(\theta)),\phi_{1}^\star(\theta)\big\rangle\\
\vdots&\ddots&\vdots\\
\langle \mathcal{F}^X_\theta(\phi_1^\star(\theta)),\phi_{p}^\star(\theta)\big\rangle &\cdots&\langle \mathcal{F}^X_\theta(\phi_p^\star(\theta)),\phi_{p}^\star(\theta)\big\rangle
\end{pmatrix}.
$$
\end{Proposition}
 
Since we do not want to assume a priori absolute summability of the filter coefficients $\Phi_\ell$, the series $\mathcal{F}^Y_\theta=({2\pi})^{-1}\sum_{h\in\mathbb{Z}}C_h^Ye^{\ii h\theta}$, where $C_h^Y=\mathrm{cov}(Y_h,Y_0)$,  may not   converge absolutely, and hence not pointwise in $\theta$. 
 As our general theory will show, the operator $\mathcal{F}^Y_\theta$ can be considered as an element of the space $L^2_{\mathbb{C}^{p\times p}}([-\pi,\pi])$, i.e.\ the collection of measurable mappings $f:[-\pi,\pi]\to \mathbb{C}^{p\times p}$ for which $\int_{-\pi}^\pi \|f(\theta)\|_F^2d\theta<~\infty$, where $\|\cdot\|_F$ denotes the Frobenius norm. Equality of $f$ and $g$ is thus understood as~$\int_{-\pi}^\pi \|f(\theta)-g(\theta)\|_F^2d\theta=0$. In particular it implies that $f(\theta)=g(\theta)$ for almost all~$\theta$.

To explain the important consequences of Proposition~\ref{pr:sd_filter}, first observe that under~\eqref{e:abssymcov}, for every frequency $\theta$, the operator $\mathcal{F}^X_\theta$ is a non-negative, self-adjoint Hilbert-Schmidt operator (see Appendix C for details). Hence, in analogy to \eqref{eq:ed}, $\mathcal{F}^X_\theta$ admits, for all $\theta$,  the spectral representation 
$$
\mathcal{F}^X_\theta(x)=\sum_{m\geq 1}\lambda_m(\theta)\langle x,\varphi_m(\theta)\rangle\varphi_m(\theta),
$$
where $\lambda_m(\theta)$ and $\varphi_m(\theta)$ denote the {\em dynamic} eigenvalues and eigenfunctions.
%
We impose the order $\lambda_1(\theta)\geq \lambda_2(\theta)\geq\ldots\geq 0$ for all $\theta \in [-\pi,\pi]$, and require that the eigenfunctions be standardized  so that $\|\varphi_m(\theta)\|=1$ for all $m\geq 1$ and $\theta\in[-\pi,\pi]$.

Assume now that we could choose the functional filters $(\phi_{m\ell}\colon \ell\in\mathbb{Z})$ in such a way that
\begin{equation}\label{eq:convl2}
\lim_{N\to\infty}\int_{-\pi}^\pi\int_0^1\left(\sum_{\ell=-N}^N\phi_{m\ell}(u)e^{\ii \ell\theta}-\varphi_m(u|\theta)\right)^2dud\theta=0.
\end{equation}
 We then have $\mathcal{F}^Y_\theta=\mathrm{diag}(\lambda_1(\theta),\ldots,\lambda_p(\theta))$ for almost all $\theta$, implying that the coordinate processes of $(Y_t)$ are uncorrelated at any lag: $\cov(Y_{mt},Y_{m^\prime s})=0$ for all~$s,t$ and  $m\neq m^\prime$. As discussed in the Introduction, this is a desirable property which the static FPCs do not possess.

\subsection{Dynamic FPCs}\label{ss:dynam}

Motivated by the discussion above, we wish to define $\phi_{m\ell}$ in such a way that~$
\phi_m^\star=\varphi_m$ (in $L^2([0,1]\times [-\pi,\pi])$).
To this end, we suppose that the function $\varphi_m(u|\theta)$ is jointly measurable in $u$ and $\theta$ (this assumption is discussed in Appendix~\ref{se:A1}). The fact that eigenfunctions are standardized to unit length implies $\int_{-\pi}^\pi\int_0^1\varphi_m^2(u|\theta) dud\theta=2\pi$. We conclude from Tonelli's theorem that $\int_{-\pi}^\pi \varphi_m^2(u|\theta)d\theta<\infty$ for almost all $u\in[0,1]$, i.e.\ that $\varphi_m(u|\theta)\in L^2([-\pi,\pi])$ for all $u\in A_m\subset [0,1]$, where $A_m$ has Lebesgue measure one. We now define, for $u\in A_m$,
\begin{equation}\label{e:phiml}
\phi_{m\ell}(u):=\frac{1}{2\pi}\int_{-\pi}^{\pi}\varphi_m(u|s)e^{-\ii\ell s}ds ;
\end{equation}
for $u\notin A_m$, $\phi_{m\ell}(u)$ is set 
 to zero.
Then, it follows from the results in Appendix~\ref{se:A1} that  \eqref{eq:convl2} holds.
We conclude that the functional filters defined via $(\phi_{m\ell}\colon\ell\in\mathbb{Z}, 1\leq m\leq p)$ belong to the class $\mathcal{C}$ and that the resulting filtered process has diagonal autocovariances at all lags.

\begin{Definition}[Dynamic functional principal components]\label{def:dfpc} Assume that $(X_t\colon t\in~\mathbb{Z})$ is a mean-zero stationary process with values in $L_H^2$ satisfying assumption \eqref{e:abssymcov}. Let~$\phi_{m\ell}$ be defined as in \eqref{e:phiml}. Then the {\em $m$-th dynamic functional principal component   score of~$(X_t)$} is\vspace{-2mm}
\begin{equation}\label{eq:dfpcs}
Y_{mt}:=\sum_{\ell \in\mathbb{Z}}\langle X_{t-\ell},\phi_{m\ell}\rangle,
\quad t\in\mathbb{Z},\>m\geq 1.\vspace{-2mm}
\end{equation}
Call $\Phi_m:=(\phi_{m\ell}\colon \ell \in\mathbb{Z})$ the {\em $m$-th dynamic FPC filter coefficients.}
\end{Definition}

\begin{Remark}\label{rem:mean}
If $EX_t=\mu$, then the dynamic FPC scores $Y_{mt}$ are defined as in \eqref{eq:dfpcs}, with $X_s$ replaced by $X_s-\mu$. 
\end{Remark}
\begin{Remark}\label{rem:unique}
Note that the dynamic scores $(Y_{mt})$ in (\ref{eq:dfpcs}) are not unique. The filter coefficients $\phi_{m\ell}$ are computed from the eigenfunctions $\varphi_m(\theta)$, which are defined up to some multiplicative factor $z$ on the complex unit circle. Hence, to be precise, we should speak of a {\em version} of $(Y_{mt})$ and a {\em version} of $(\phi_{m\ell})$. We  further discuss this issue after Theorem~\ref{th:inversion} and in Section~\ref{ss:consistency}.
\end{Remark}



The rest of this section is devoted to some important properties of dynamic FPCs.
\begin{Proposition}[Elementary properties]\label{pr:elementary}
Let  $(X_t\colon t\in\mathbb{Z})$ be a real-valued stationary process satisfying \eqref{e:abssymcov}, with  dynamic FPC scores  $Y_{mt}$.   Then,\\[1ex]
(a) the eigenfunctions $\varphi_m(\theta)$ are Hermitian, and hence $Y_{mt}$ is real;\\[1ex]
(b) if $C_h=0$ for $h\neq 0$, the dynamic FPC scores coincide with the static ones.
\end{Proposition}

\begin{Proposition}[Second-order properties]\label{pr:secondorder}
Let  $(X_t\colon t\in\mathbb{Z})$ be a stationary process satisfying \eqref{e:abssymcov}, with  dynamic FPC scores  $Y_{mt}$. Then,\\[1ex]
(a) the series defining $Y_{mt}$ is mean-square convergent, with
$$EY_{mt}=0\quad\text{and}\quad EY_{mt}^2=\sum_{\ell\in\mathbb{Z}}\sum_{k\in\mathbb{Z}}\langle C_{\ell-k}(\phi_{m\ell }),\phi_{mk}\rangle;\vspace{-2mm}$$
(b)  the dynamic FPC scores $Y_{mt}$ and $Y_{m^\prime s}$ are uncorrelated for all $s,t$ and   $m\neq m^\prime$. In other words, if $Y_t=(Y_{1t},\ldots, Y_{pt})^\prime$ denotes some $p$-dimensional score vector and~$C^Y_h$ its lag-$h$ covariance matrix, then $C^Y_h$ is diagonal for all $h$;\\[1ex]
(c) the long-run covariance matrix of the dynamic FPC score vector process $(Y_t)$ is 
$$
\lim_{n\to\infty}\frac{1}{n}\Var(Y_{1}+\cdots+Y_{n})=2\pi\,\mathrm{diag}(\lambda_1(0),\ldots,\lambda_p(0)).
$$
\end{Proposition}

The next theorem, which  tells us how  the original process $(X_{t}(u)\colon t\in\mathbb{Z},\, u\in [0,1])$ can be recovered 
from $(Y_{mt}\colon t\in\mathbb{Z},\,m\geq 1)$,  is the dynamic analogue of the static Karhunen-Lo\`eve expansion \eqref{eq:KLexpansion} associated with static principal components.

\begin{Theorem}[Inversion formula]\label{th:inversion}
Let $Y_{mt}$ be the dynamic FPC scores related to the process
$(X_{t}(u)\colon t\in\mathbb{Z},\, u\in [0,1])$. Then,
\begin{equation}\label{e:dynKL}
X_t(u)=\sum_{m\geq 1}X_{mt}(u)\quad\text{with}\quad X_{mt}(u):=\sum_{\ell \in\mathbb{Z}}Y_{m,t+\ell}\phi_{m\ell}(u)\vspace{-2mm}
\end{equation}
(where convergence is in mean square).
Call \eqref{e:dynKL} the {\em dynamic Karhunen-Lo\`eve expansion of $X_t$}.
\end{Theorem}

We have mentioned in Remark~\ref{rem:unique} that dynamic FPC scores are not unique.  In contrast, our proofs show  that the curves $X_{mt}(u)$ are unique. To get some intuition, let us draw a simple analogy to the static case. There,  each $v_\ell$ in the Karhunen-Lo\`eve expansion \eqref{eq:KLexpansion} can be replaced by $-v_\ell$, i.e., the FPCs are defined up to their signs. The $\ell$-th score is $\langle X,v_\ell\rangle$ or $\langle X,-v_\ell\rangle$, and thus is not unique either. However, the curves $\langle X,v_\ell\rangle v_\ell$ and $\langle X,-v_\ell\rangle (-v_\ell)$ are identical.

The sums $\sum_{m=1}^pX_{mt}(u)$, $p\geq 1$, can be seen as $p$-dimensional reconstructions of~$X_t(u)$, which only involve the $p$ time series $(Y_{mt}\colon t\in\mathbb{Z})$, $1\leq m\leq p$.
Competitors to this reconstruction are obtained by replacing $\phi_{m\ell}$ in~\eqref{eq:dfpcs} and \eqref{e:dynKL} with alternative  sequences $\psi_{m\ell}$ and $\upsilon_{m\ell}$.   The next theorem shows that, among all filters in $\mathcal{C}$, the  dynamic Karhunen-Lo\`eve expansion  (\ref{e:dynKL}) 
  approximates $X_t(u)$ in an optimal way. 
\begin{Theorem}[Optimality of Karhunen-Lo\`eve expansions]\label{th:optimality}
Let $Y_{mt}$ be the dynamic FPC scores related to the process
$(X_{t}\colon t\in\mathbb{Z})$, and define $X_{mt}$   as in Theorem~\ref{th:inversion}.
Let  $\tilde X_{mt}=\sum_{\ell\in Z} \tilde Y_{m,t+\ell}\,\upsilon_{m\ell}$, with $\tilde Y_{mt}=\sum_{\ell\in Z}\langle X_{t-\ell},\psi_{m\ell}\rangle$, where $(\psi_{mk}\colon k\in\mathbb{Z})$ and~$(\upsilon_{mk}\colon k\in\mathbb{Z})$ are sequences in $H$ belonging to $\mathcal{C}$.  Then,\vspace{-2mm}
\begin{equation}\label{eq:dynopt}
E\|X_t-\sum_{m=1}^pX_{mt}\|^2=\sum_{m>p}\int_{-\pi}^\pi\lambda_m(\theta)d\theta\leq E\|X_t-\sum_{m=1}^p\tilde X_{mt}\|^2\quad\forall p\geq 1.\vspace{-2mm}
\end{equation}
\end{Theorem}

Inequality \eqref{eq:dynopt} can be interpreted as the dynamic version of \eqref{eq:fpcopt}.
Theorem~\ref{th:optimality} also suggests the proportion 
\begin{equation}\label{e:propvar}
\sum_{m\leq p}\int_{-\pi}^\pi\lambda_m(\theta)d\theta\Big\slash E\|X_1\|^2
\end{equation}
of variance explained by the first $p$ dynamic FPCs as a natural measure of how well a functional time series can be represented in  dimension $p$.  

\subsection{Estimation and asymptotics}\label{ss:consistency}

In practice,  dynamic FPC scores need to be calculated from an estimated version of $\mathcal{F}^X_\theta$.  At the same time, the infinite series defining the scores need to be replaced by finite approximations. Suppose again that $(X_t : t \in \mathbb{Z})$ is a weakly stationary zero-mean time series such that \eqref{e:abssymcov} holds.
Then, a natural estimator for $Y_{mt}$ is
\begin{align}\label{eq:estimator:fullyobs}
\hat Y_{mt}:=\sum_{\ell = -L}^L\langle X_{t-\ell},\hat \phi_{m\ell}\rangle,\quad m=1,\ldots,p\quad\text{and}\quad t=L+1,\ldots n-L,
\end{align}
where $L$ is some integer and $\hat \phi_{m\ell}$ is computed  from some estimated spectral density operator $\mathcal{\hat F}^X_\theta$.  For the latter, we impose the following preliminary assumption. 
\begin{assumption}{\bf{B.1}}\label{assumption:consistent-estimator}
The estimator $\mathcal{\hat F}^X_\theta$ is consistent in integrated mean square, i.e. 
\begin{align}\label{integrated:expectation}
\int_{-\pi}^\pi E \| \mathcal{F}_\theta^X - \mathcal{\hat F}^X_\theta \|^2_\cS\, d\theta \rightarrow 0\quad \text{as}\quad n\to\infty.
\end{align}
\end{assumption}

Panaretos and Tavakoli~\cite{panaretos:tavakoli:2012} propose an estimator $\mathcal{\hat F}^X_\theta$   satisfying  \eqref{integrated:expectation} under certain functional cumulant conditions. By stating~\eqref{integrated:expectation} as an assumption, we intend to keep the theory more widely applicable. For example, the following proposition shows that  estimators satisfying Assumption~B.1 also exist under $L^4$-$m$-{\it approximability}, a dependence concept for functional data introduced in H\"ormann and Kokoszka~\cite{hormann:kokoszka:2010}. Define
\begin{equation}\label{Hoer-Koko}
\mathcal{\hat F}_\theta^X=\sum_{|h|\leq q}\left(1-\frac{|h|}{q}\right)\hat C_h^X e^{-\ii h\theta},\quad 0<q<n,
\end{equation}
where $\hat C_h^X$ is the usual empirical autocovariance operator at lag $h$.
\begin{Proposition}\label{lem:gammas}
Let $(X_t : t \in \mathbb{Z})$ be  $L^4$-$m$-approximable, and let $q=q(n)\to\infty$ such that $q^3=o(n)$. Then the estimator $\mathcal{\hat F}_\theta^X$ defined in (\ref{Hoer-Koko}) satisfies Assumption~B.1. The approximation error is $O(\alpha_{q,n})$, where
$$
\alpha_{q,n}=\frac{q^{3/2}}{\sqrt{n}}+\frac{1}{q}\sum_{|h|\leq q}|h|\|C_h\|_\mathcal{S}+\sum_{|h|>q}\|C_h\|_\mathcal{S}.
$$
\end{Proposition}
\begin{Corollary}
Under the assumptions of Proposition~\ref{lem:gammas} and $\sum_{h}|h|\|C_h\|_\mathcal{S}<\infty$ the convergence rate of the estimator \eqref{Hoer-Koko} is $O(n^{-1/5})$.
\end{Corollary}

Since our method requires the estimation of eigenvectors of the spectral density operator, we also need to introduce certain identifiability constraints on eigenvectors. Define $\alpha_1(\theta):= \lambda_1(\theta) - \lambda_2(\theta)$ and 
$$\alpha_m(\theta):= \min\{ \lambda_{m-1}(\theta) - \lambda_{m}(\theta),\lambda_{m}(\theta) - \lambda_{m+1}(\theta) \}\quad\text{for}\quad  m>1,$$
 where $\lambda_i(\theta)$ is the $i$-th largest eigenvalue of the spectral density operator evaluated in $\theta$. 
\begin{assumption}{\bf{B.2}}  For all $m$, $\alpha_m(\theta)$ has  finitely many zeros.
\end{assumption}
Assumption~B.2 essentially guarantees disjoint eigenvalues for all $\theta$. It is a very common assumption in functional PCA, as it ensures that eigenspaces are one-dimensional, and thus eigenfunctions are unique up to their signs. To guarantee identifiability, it only remains to provide a rule for choosing the signs. In our context,  the situation is slightly more complicated, since we are working in a complex setup. The eigenfunction $\varphi_m(\theta)$ is unique up to multiplication by a number on the complex unit circle. A possible way to fix the direction of the eigenfunctions is to impose a constraint of the form $\langle \varphi_m(\theta),v\rangle\in (0,\infty)$ for some given function $v$. In other words, we choose the orientation of the eigenfunction such that its inner product with some reference curve $v$ is a positive real number. This rule identifies $\varphi_m(\theta)$, as long as it is not orthogonal to $v$. The following assumption ensures that such  identification is possible on a large enough set of frequencies $\theta\in[-\pi,\pi]$. 
\begin{assumption}{\bf{B.3}} Denoting by $\varphi_m(\theta)$ be the $m$-th dynamic eigenvector of $\mathcal F_\theta^X$,  there exists~$v$  such that $\langle\varphi_m(\theta),v\rangle\neq 0$ for almost all~$\theta\in [-\pi,\pi]$.  \end{assumption}
From now on, we tacitly assume that the orientations of $\varphi_m(\theta)$ and $\hat\varphi_m(\theta)$ are chosen so that  $\langle \varphi_m(\theta),v\rangle$ and $\langle \hat\varphi_m(\theta),v\rangle$ are in $[0,\infty)$ for almost all $\theta$. Then, we have the following result.

\begin{Theorem}[Consistency]\label{th:dfpc:consistency}
Let $\hat Y_{mt}$ be the random variable defined by \eqref{eq:estimator:fullyobs} and suppose that Assumptions $B.1$--$B.3$ hold. Then, for some sequence $L=L(n)\to\infty$, we have
$\hat Y_{mt}\convP Y_{mt}$ as $n\to\infty$.
\end{Theorem}

Practical guidelines for the choice of $L$ are given in the next section. 

\section{Practical implementation}\label{se:pract}

In applications, data can only be recorded discretely. A curve $x(u)$ is observed on grid points $0\leq u_1<u_2<\cdots<u_r\leq 1$. Often, though not necessarily so, $r$ is very large (high frequency data). The sampling frequency $r$ and the sampling points $u_i$ may change from observation to observation. Also, data may be recorded with or without measurement error, and time warping (registration) may be required.  For deriving limiting results, a common assumption is that $r\to\infty$, while a possible measurement error tends to zero. All these specifications have been extensively studied in the literature, and we omit here the technical exercise to cast our theorems and propositions in one of these setups. Rather, we show how to implement the proposed method, after the necessary preprocessing steps have been carried out. Typically, data are then represented in terms of a finite (but possibly large) number of basis functions $(v_k\colon 1\leq k\leq d)$, i.e., $x(u)=\sum_{k=1}^d x_kv_k(u)$. Usually Fourier bases, $b$-splines or wavelets are used. For an excellent survey on preprocessing the raw data, we refer to Ramsey and Silverman~\cite[Chapters 3--5]{ramsay:silverman:2005}.

In the sequel, we write $(a_{ij}\colon 1\leq i,j\leq d)$ for a $d\times d$ matrix with entry~$a_{ij}$ in row~$i$ and column~$j$.
Let $x$ belong to the span $ H_d:=\spa (v_k\colon 1\leq k\leq d)$ of~$v_1,\ldots ,v_d$. Then $x$ is of the form $\bv^{\prime}\bx$, where $\bv=(v_1,\ldots, v_d)^{\prime}$ and $\bx=(x_1,\ldots,x_d)^{\prime}$. We assume that the basis functions $v_1,\ldots ,v_d$ are linearly independent, but they need not be orthogonal. Any statement about $x$  can be expressed as an equivalent statement about $\bx$. In particular,  if $A:H_d\to H_d$ is a linear operator, then, for $x\in H_d$,\vspace{-1mm}
$$
A(x)=\sum_{k=1}^d x_kA(v_k)=\sum_{k=1}^d\sum_{k'=1}^d x_k\langle A(v_k),v_{k'}\rangle v_{k'}=\bv^{\prime}\mathfrak{A}\bx ,
\vspace{-1mm}$$
where $\mathfrak{A}'=(\langle A(v_i),v_j\rangle\colon 1\leq i,j\leq d)$.
Call $\mathfrak{A}$ the {\em corresponding matrix} of $A$ and~$\bx$ the {\em corresponding vector} of $x$.

The following simple results are stated without proof. 
\begin{lemma}\label{le:corresp}
Let $A,B$ be linear operators on $H_d$, with corresponding matrices $\mathfrak{A}$ and~$\mathfrak{B}$, respectively. Then,\\[1ex]
\noindent
(i) for any $\alpha,\beta \in \mathbb{C}$, the corresponding matrix of $\alpha A+\beta B$ is $\alpha \mathfrak{A}+\beta \mathfrak{B}$;\\[1ex]
\noindent
(ii) $A(e)=\lambda e$ iff $\mathfrak{A}\mathbf{e}=\lambda\mathbf{e}$, where $e=\bv^{\prime}\mathbf{e}$;\\[1ex]
\noindent
(iii) letting $A:= \Sum_{i=1}^p\Sum_{j=1}^p g_{ij} v_i \otimes v_j$,  $G:= (g_{ij}\colon 1 \leq i,j \leq d)$, where $g_{ij} \in \mathbb{C}$, and $V:= (\langle v_i,v_j\rangle\colon 1 \leq i,j \leq d)$,
the corresponding matrix of $A$ is $\mathfrak{A}=G V^{\prime}$. 
\end{lemma}

To obtain the corresponding matrix of the spectal density operators $\SPDO_\theta^X$, first observe that, if $X_k=\sum_{i=1}^dX_{ki}v_i=:\bv^{\prime}\bX _k$, then\vspace{-1mm}
$$
C_h^X=EX_h\otimes X_{0}=\sum_{i=1}^d \sum_{j=1}^dEX_{hi}X_{0j}v_i\otimes v_j.
\vspace{-1mm}$$
 It follows from Lemma~\ref{le:corresp} (iii) that $\mathfrak{C}_h^X=C_h^\bX V^{\prime}$  is the corresponding matrix of $C_h^\bX:=E\bX _h\bX _0'$;  the linearity property (i) then implies that  \vspace{-1mm}  
\begin{equation}\label{eq:sd_matrix}
\mathfrak{F}_\theta^X=\frac{1}{2\pi}\Big(\sum_{h\in\mathbb{Z}}C_h^\bX e^{-\ii h\theta}\Big)V^{\prime}
\vspace{-1mm}\end{equation}
is the corresponding matrix of $\SPDO_\theta^X$. Assume that $\lambda_m(\theta)$ is the $m$-th largest eigenvalue of $\mathfrak{F}_\theta^X$, with eigenvector $\boldsymbol{\varphi}_m(\theta)$. Then $\lambda_m(\theta)$ is also an eigenvalue of $\SPDO^X_\theta$ and $\bv^{\prime}\boldsymbol{\varphi}_m(\theta)$ is the corresponding eigenfunction, from which we can compute, via its Fourier expansion, the dynamic FPCs. In particular, we have
$$
\phi_{mk}=\frac{\bv^{\prime}}{2\pi}\int_{-\pi}^\pi\boldsymbol{\varphi}_m(s)e^{-\ii k s}ds=:\bv^{\prime}\bphi_{mk},
$$
and hence
\begin{equation}\label{eq:finitedfpc}
Y_{mt}=\sum_{k\in\mathbb{Z}}\int_0^1 \bX_{t-k}'\bv(u)\bv^{\prime}(u) \bphi_{mk}du=\sum_{k\in\mathbb{Z}}\bX_{t-k}' V\bphi_{mk}.
\end{equation}

%
%

In view of \eqref{eq:sd_matrix}, our task is now to replace the spectral density matrix\vspace{-1mm}
$$
\SPDO_\theta^\bX=\frac{1}{2\pi}\sum_{h\in\mathbb{Z}}C_h^\bX e^{-\ii h\theta}
\vspace{-1mm}$$
of the coefficient sequence $(\bX_k)$ by some estimate. For this purpose, we can use existing multivariate techniques. Classically, we would put, for $|h|<n$, \vspace{-1mm}
$$
\hat C_h^\bX:=\frac{1}{n}\sum_{k=h+1}^{n}\bX _{k}\bX _{k-h}',\quad h\geq 0,
\quad\text{and}\quad
\hat C_h^\bX:=\hat C_{-h}^\bX,\quad h< 0
\vspace{-1mm}$$
(recall that we throughout assume that the data are centered) and use, for example, some lag window estimator \vspace{-1mm}
\begin{align}\label{e:bartlett_estimator}
\hat\SPDO_\theta^\bX:=\frac{1}{2\pi}\sum_{|h|\leq q}w(h/q)\hat C_h^\bX e^{-\ii h\theta},
\vspace{-1mm}\end{align}
where $w$ is some appropriate weight function, $q=q_n\to\infty$ and $q_n/n\to 0$. For more details concerning common choices of $w$ and the tuning parameter $q_n$, we refer to Chapters~10--11 in Brockwell and Davis~\cite{brockwell:davis:1990} and to Politis~\cite{politis:2011}. We then set $\hat{\mathfrak{F}}_\theta^X:=\hat\SPDO_\theta^\bX V^{\prime}$ and compute the eigenvalues and eigenfunctions~$\hat\lambda_m(\theta)$ and $\hat{\boldsymbol{\varphi}}_m(\theta)$ thereof, which serve as estimators of $\lambda_m(\theta)$ and $\boldsymbol{\varphi}_m(\theta)$, respectively. We estimate the filter coefficients by $\hat\phi_{mk}=\frac{\bv^{\prime}}{2\pi}\int_{-\pi}^\pi\hat{\boldsymbol{\varphi}}_m(s)e^{\ii ks}ds$. Usually, no analytic form of~$\hat{\boldsymbol{\varphi}}_m(s)$ is available, and one has to perform numerical integration. We take the simplest approach, which is to set\vspace{-1mm}
$$
\hat\phi_{mk}=\frac{\bv^{\prime}}{2\pi(2N_\theta+1)}\sum_{j=-N_\theta}^{N_\theta}\hat{\boldsymbol{\varphi}}_m(\pi j/N_\theta)e^{\ii ks}=:\bv^{\prime}\hat{\bphi}_{mk},\quad (N_\theta\gg 1).
\vspace{-1mm}$$
The larger $N_\theta$ the better. This clearly depends on the available computing power.

Now, we substitute~$\hat{\bphi}_{mk}$ into~\eqref{eq:finitedfpc}, replacing the infinite sum with a rolling window\vspace{-1mm}
\begin{align}\label{e:est_scores_equation}
\hat Y_{mt}=\sum_{k=-L}^L \bX_{t-k}' V\hat{\bphi}_{mk}. 
\vspace{-1mm}\end{align}
This expression only can be computed for $t\in\{L+1,\dots,n-L\}$; for  
 $1\leq t\leq L$ or~$n-L+1\leq t\leq n$,   set $X_{-L+1}=\cdots=X_0=X_{n+1}=\cdots=X_{n+L}=EX_1=0$. This, of course, creates a certain bias  on the boundary of the observation period. As for the choice of $L$, we observe that $\sum_{\ell\in\mathbb{Z}}\|\hat\phi_{m\ell}\|^2=1$. It is then natural to choose~$L$ such that $\sum_{-L\leq \ell\leq L}\|\hat\phi_{m\ell}\|^2\geq 1-\epsilon$, for some small threshold $\epsilon$, e.g., $\epsilon=0.01$.
 
Based on this definition of  $\hat\phi_{mk}$, 
   we obtain an empirical $p$-term dynamic Karhunen-Lo\`eve expansion
\begin{align}\label{e:est_KL_equation}
\hat X_{t}=\sum_{m=1}^p\sum_{k=-L}^L \hat Y_{m,t+k}\hat\phi_{mk},\   \text{with}\  \  \hat Y_{mt}=0, \ t\in\{-L+1,\ldots,0\}\cup\{n+1,\ldots,n+L\}.
\end{align}

Parallel to \eqref{e:propvar}, the proportion of variance explained by the first $p$ dynamic FPCs can be estimated through \vspace{-2mm}
$$
\mathrm{PV}_{\mathrm{dyn}}(p):=\frac{\pi}{N_\theta}\sum_{m\leq p}\sum_{j=-N_\theta}^{N_\theta}\hat\lambda_m(\pi j/N_\theta)\Big\slash \frac{1}{n}\sum_{k=1}^n\|X_k\|^2.
 \vspace{-2mm}$$
We will use $(1-\mathrm{PV}_{\mathrm{dyn}}(p))$ as a measure of the loss of information incurred when considering a dimension reduction to dimension $p$. Alternatively, one also can use the normalized mean squared error \vspace{-2mm}
\begin{equation}\label{e:est_propvar2}
\mathrm{NMSE}(p):=\sum_{k=1}^{n}\|X_k-\hat X_k\|^2\Big\slash \sum_{k=1}^{n}\|X_k\|^2.
 \vspace{-2mm}\end{equation}
Both quantities converge to the same limit.

\section{A real-life illustration}\label{se:reald}

In this section, we draw a comparison between dynamic and static FPCA on basis of a real data set.
The observations are half-hourly measurements of the concentration (measured in $\mu gm^{-3}$) of particulate matter with an aerodynamic diameter of less than $10\mu m$, abbreviated as {\tt PM10}, in ambient air taken in Graz, Austria from October 1, 2010 through March 31, 2011. Following Stadlober et al.~\cite{stadlober:hoermann:pfeiler:2008} and Aue et al.~\cite{aue:dubart:hoermann:2012}, a square-root transformation was performed  in order to stabilize the variance and avoid heavy-tailed observations. Also, we removed some outliers and a seasonal (weekly) pattern induced from different traffic intensities on business days and weekends. Then we use the software {\tt R} to transform the raw data, which is discrete,  to functional data, as explained in Section~\ref{se:pract}, using 15 Fourier basis functions. The resulting curves for 175 daily  observations, $X_1,\ldots,X_{175}$, say,  roughly representing one winter season, for which pollution levels are known to be high, are displayed in Figure~\ref{fig:pm_curves}. 

\begin{figure}[ht]
\centering
\includegraphics[width=8cm]
{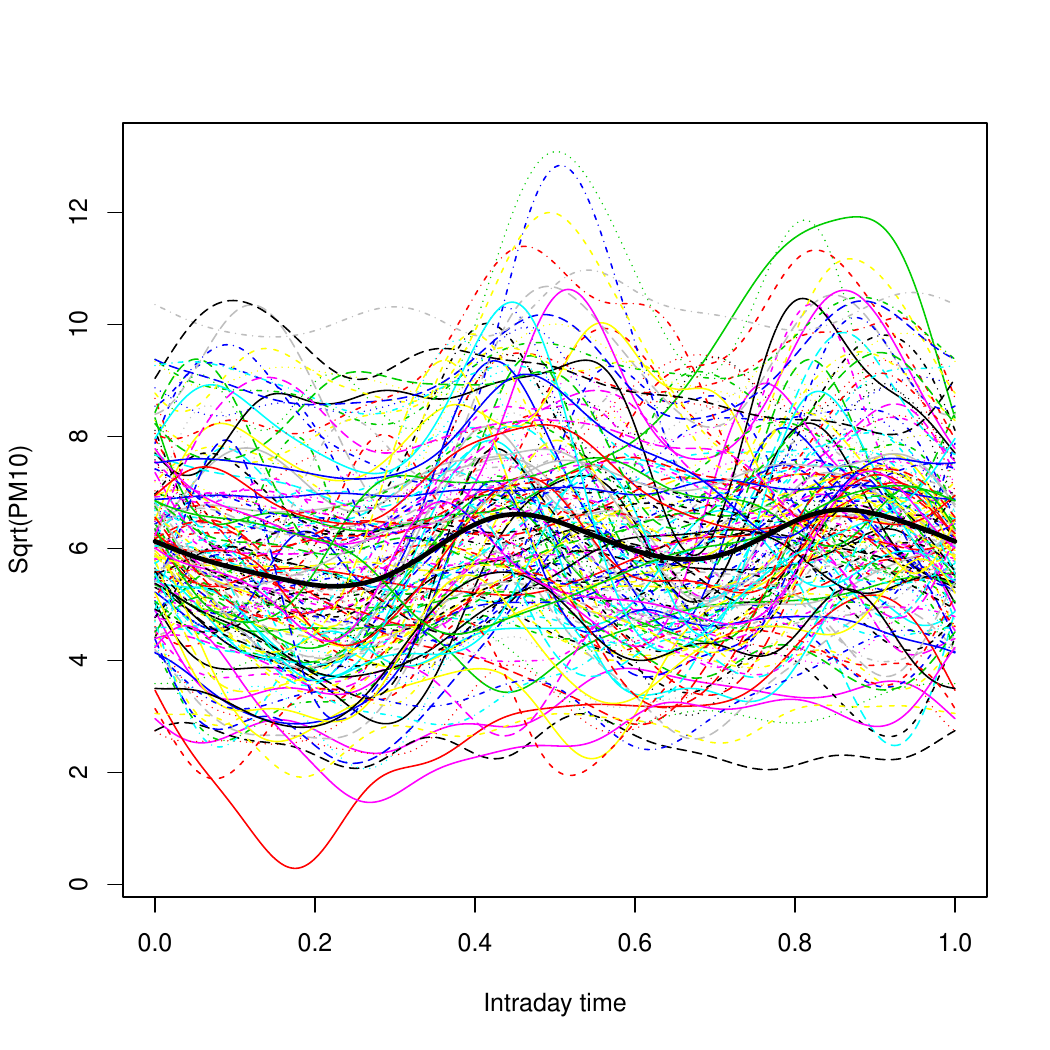}
\caption{A plot of  175 daily curves $x_t(u)$, $1\leq t\leq 175$, where $x_t(u)$ are the square-root transformed and detrended   functional observations of {\tt PM10}, based on  15 Fourier basis functions. The solid black line represents the sample mean curve~$\hat\mu(u)$.}
\label{fig:pm_curves}
\end{figure}

From those data, we computed the (estimated) first dynamic FPC score sequence $(\hat Y_{1t}^{\mathrm{dyn}}\colon~1\leq~t\leq~175)$.  To this end, we centered the data at their empirical mean $\hat\mu(u)$, then implemented the procedure described in Section~\ref{se:pract}. We used the traditional Bartlett kernel $w(x)=1-|x|$ in \eqref{e:bartlett_estimator} to obtain an estimator for the spectral density operator, with   bandwidth   $q=\lfloor  {n}^{1/2}\rfloor=13$. More sophisticated estimation methods, as those proposed,  for example,  by Politis~\cite{politis:2011}, of course can be considered; but they also depend on additional tuning parameters, still leaving much of the selection  to the practitioner's choice. From $\hat{\mathcal F}_\theta^\bX$ we obtain the estimated filter elements $\hat\phi_{1\ell}$. It turns out that they fade away quite rapidly. In particular $\sum_{\ell=-10}^{10}\|\hat\phi_{1\ell}\|^2\approx 0.998$. Hence, for calculation of the scores in \eqref{e:est_scores_equation} it is justified to choose $L=10$. The five central filter elements $\hat\phi_{1\ell}(u)$, $\ell=-2,\ldots,2$, are plotted in Figure~\ref{fig:filters}.

\begin{figure}[ht]
\centering
\includegraphics[width=15cm,height=5cm]
{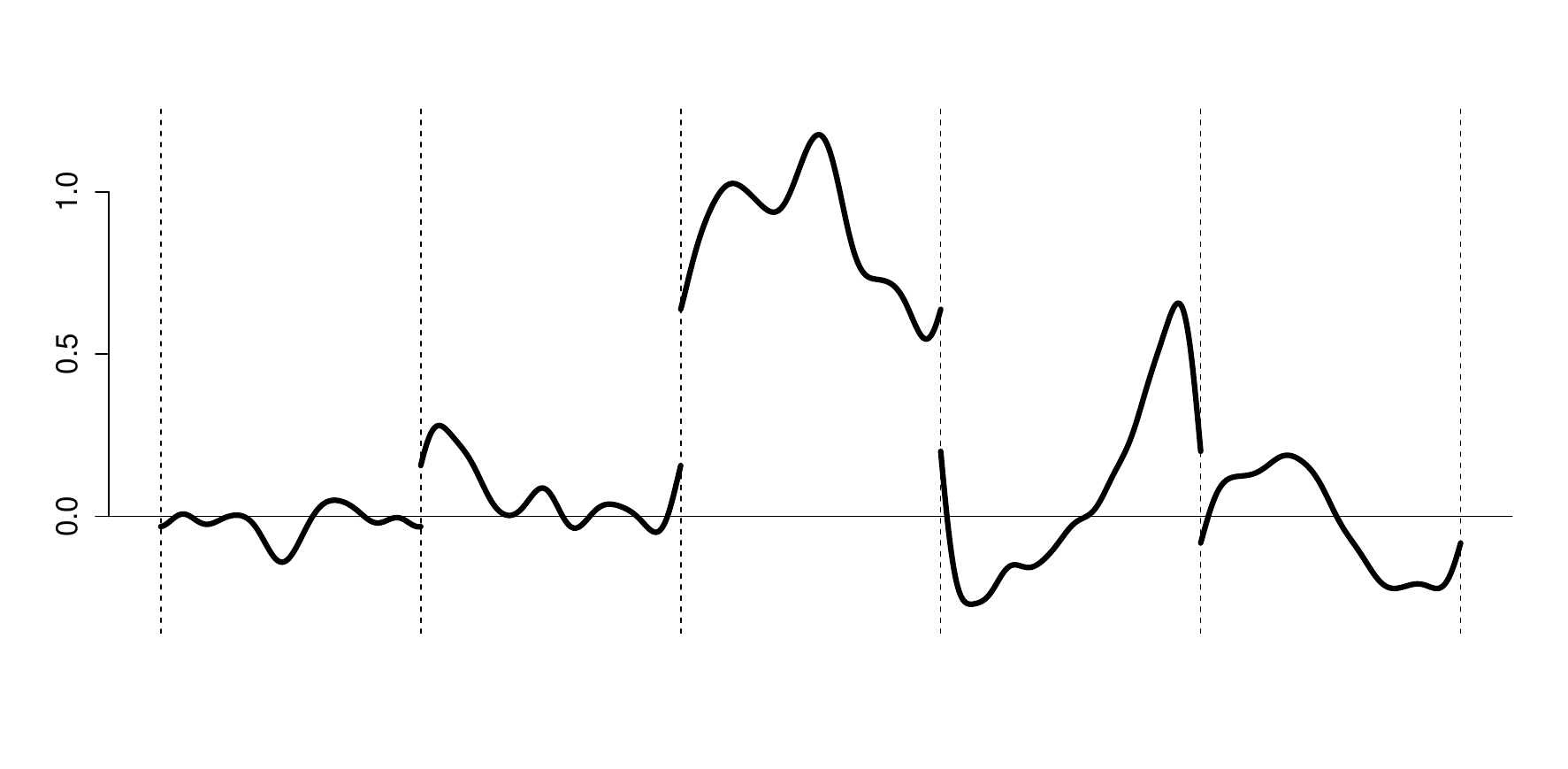}
\caption{The five central filter elements $\hat\phi_{1,-2}(u),\ldots,\hat\phi_{1,2}(u)$ (from left to right).}
\label{fig:filters}
\end{figure}

 Further components could be computed similarly, but for 
the purpose of demonstration we  focus   on one component only. In fact, the first dynamic FPC already explains about $80\%$ of the total variance, compared to the $73\%$ explained by the first static FPC.  The latter was also computed, resulting in the static FPC score sequence $(\hat Y_{1t}^{\mathrm{stat}}\colon~1\leq~t\leq~175)$.
Both sequences are shown in Figure~\ref{fig:pm_1fpc}, along with their differences.  

\begin{figure}[htbp]
\centering
\includegraphics[width=5cm]{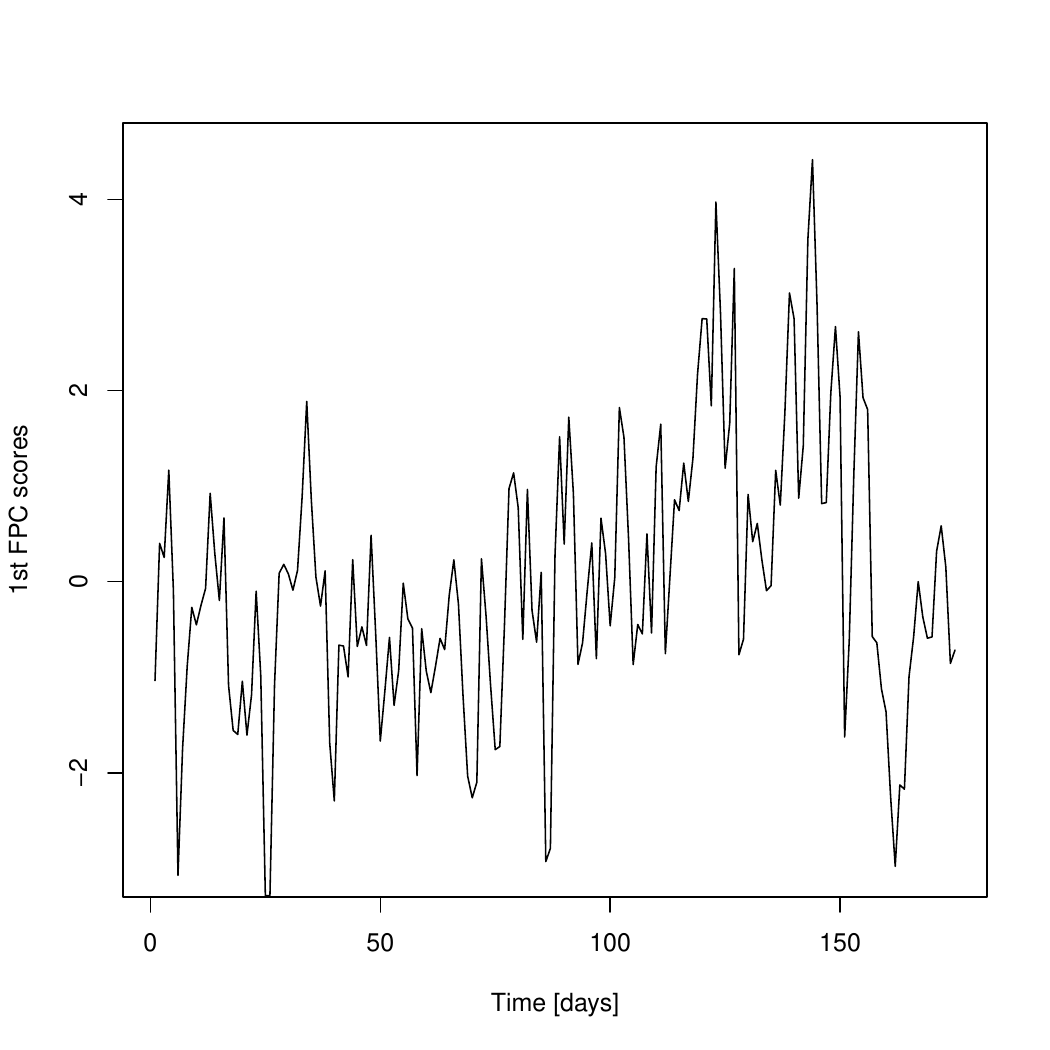}\hfill
\includegraphics[width=5cm]{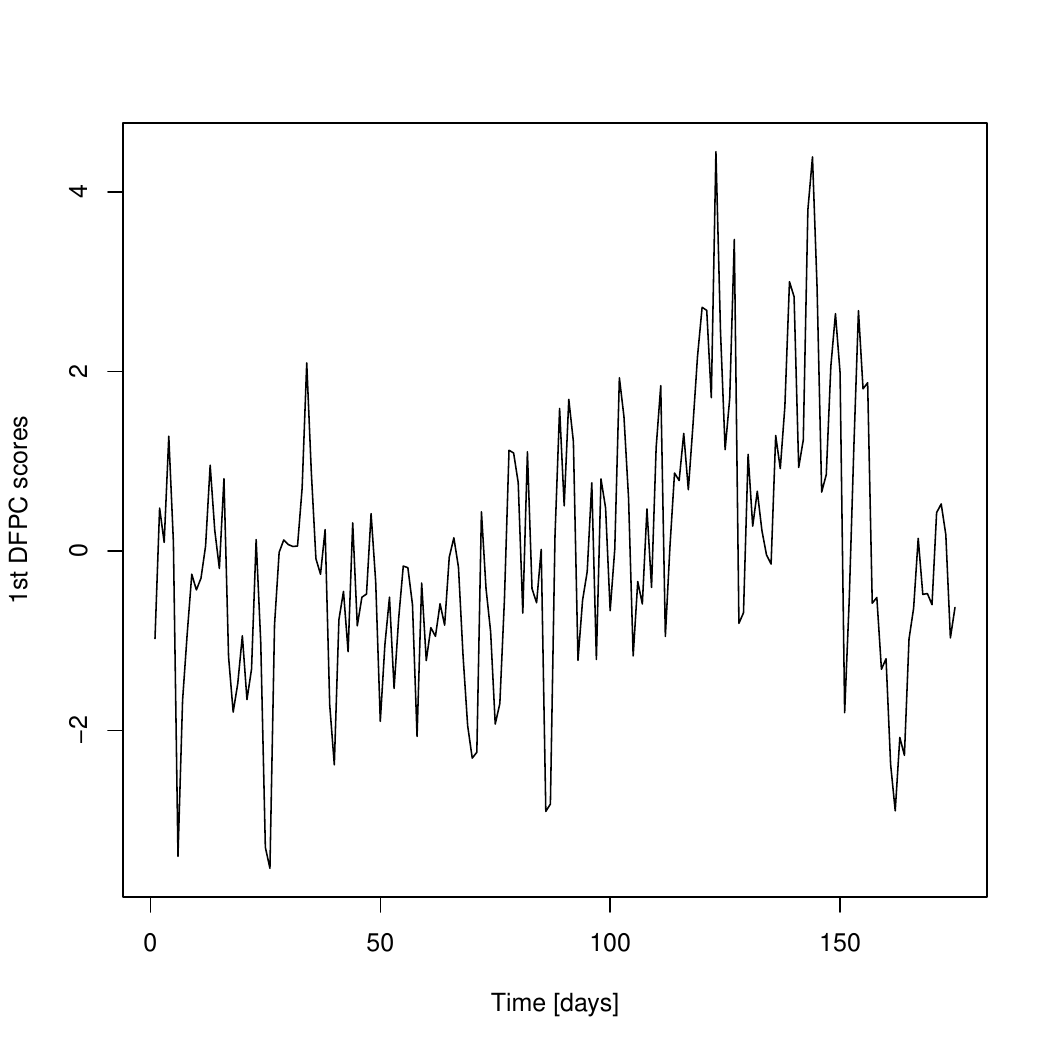}\hfill
\includegraphics[width=5cm]{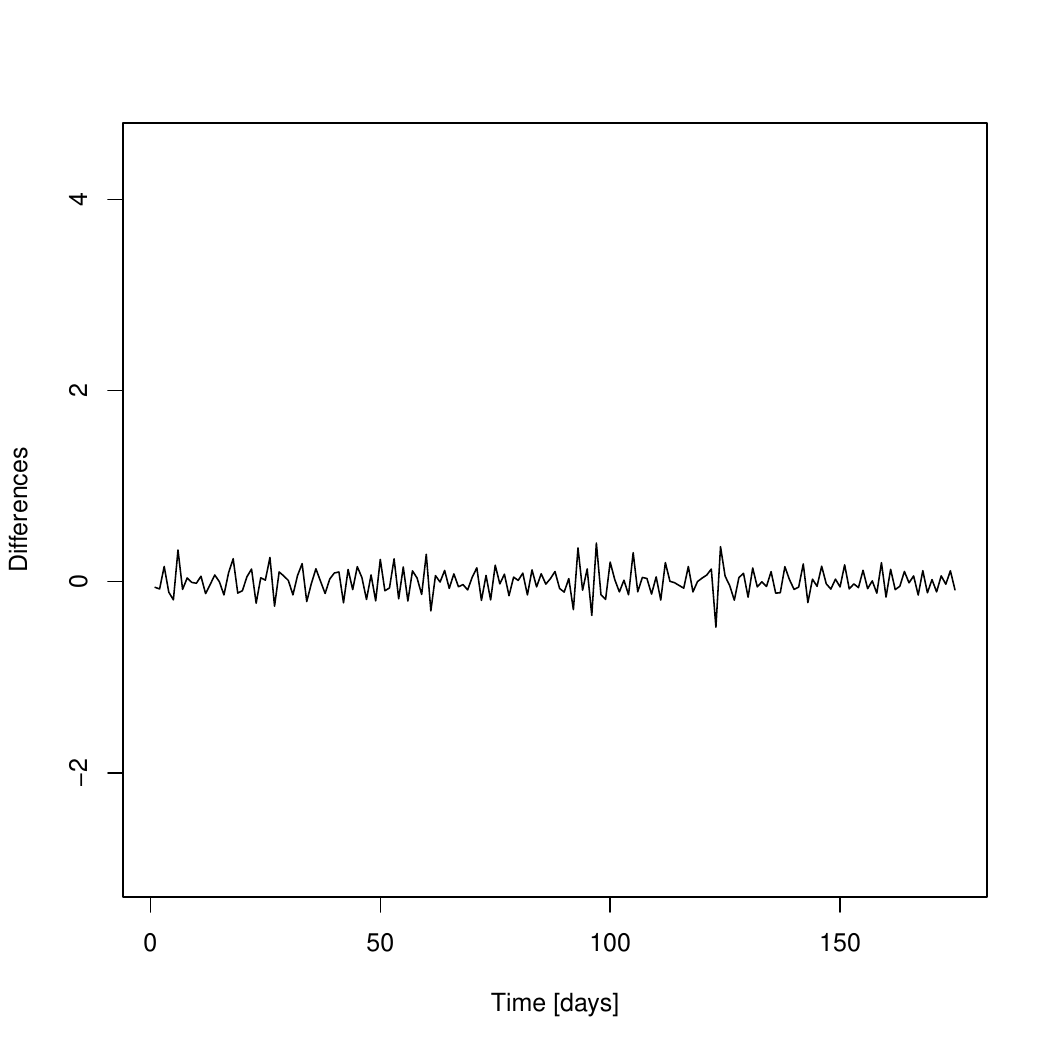}\vspace{-2mm}
\caption{  First static 
 (left panel) and first  dynamic  (middle panel) FPC score  sequences, 
and their differences 
 (right panel).\vspace{-2mm}}
\label{fig:pm_1fpc}
\end{figure}

Although based on entirely different ideas,  the static and dynamic scores in Figure~\ref{fig:pm_1fpc} (which, of course, are not loading the same functions) 
appear to be remarkably close to one another. The reason why the dynamic Karhunen-Lo\`eve expansion  accounts for a significantly larger amount of the total variation is that, contrary to its static counterpart,  it does not just involve the present observation. 
%
%

To get more statistical insight into those results, let us consider the first static sample FPC, $\hat v_1(u)$, say, displayed in Figure~\ref{fig:pm_pca}.
\begin{figure}[ht]
\centering
\includegraphics[width=13cm]{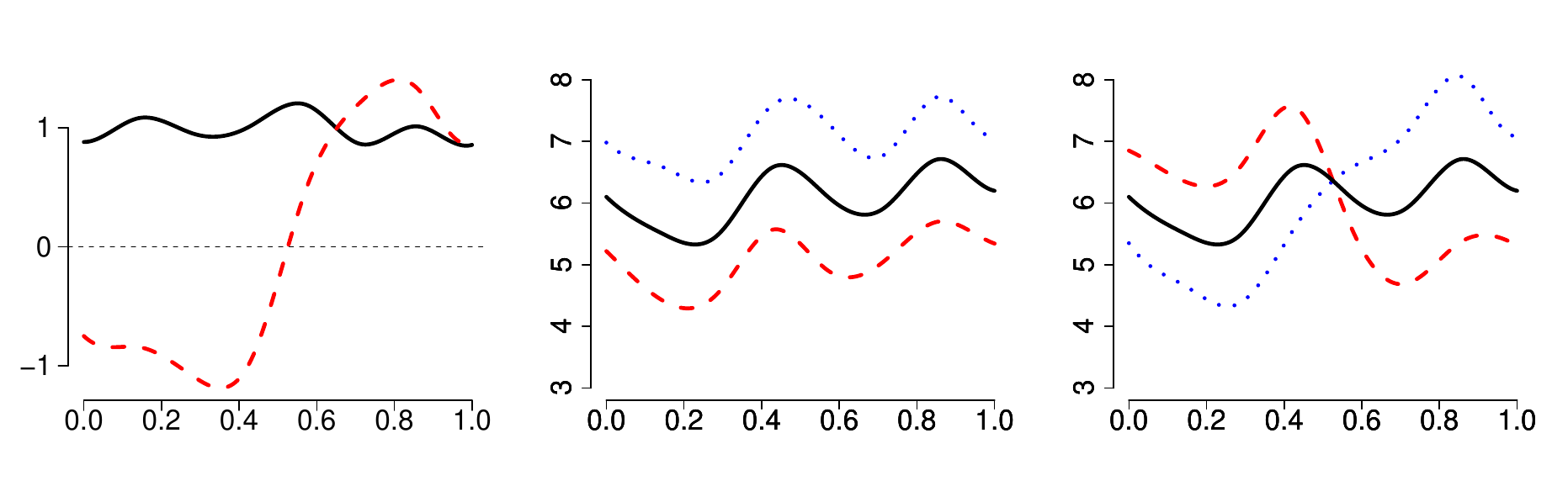}\vspace{-2mm}
\caption{First static FPC $\hat v_1(u)$ (solid line), and second static FPC $\hat v_2(u)$ (dashed line) [left panel]. $\hat\mu(u)\pm \hat v_1(u)$  [middle panel] and $\hat\mu(u)\pm \hat v_2(u)$  [right panel] describe the effect of the first and second static FPC on the mean curve.}
\label{fig:pm_pca}
\end{figure}
 We see that $\hat v_1(u)\approx 1$ for all $u\in[0,1]$, so that the static FPC score $\hat Y_{1t}^\mathrm{stat}=\int_0^1(X_t(u)-\hat\mu(u))\hat v_1(u)du$ \ roughly coincides with the average deviation of $X_t(u)$ from the sample mean $\hat\mu(u)$: the effect of a large (small) first score corresponds to a large (small) daily average of $\sqrt{{\tt PM10}}$. In view of the similarity between  $\hat Y_{1t}^\mathrm{dyn}$ and~$\hat Y_{1t}^\mathrm{stat}$, it is possible to attribute the same interpretation to the dynamic FPC scores. However, regarding the dynamic Karhunen-Lo\`eve expansion, dynamic FPC scores should be interpreted sequentially. To this end, let us take advantage of the fact that  $\sum_{\ell=-1}^{1}\|\hat\phi_{1\ell}\|^2\approx 0.92$. In the approximation by a single-term dynamic Karhunen-Lo\`eve expansion, we  thus roughly have
$$
X_t(u)\approx \hat\mu(u)+\sum_{\ell=-1}^1\hat Y_{1,t+\ell}^\mathrm{dyn}\hat \phi_{1\ell}(u).
$$ 
This suggests   studying the impact of triples $(\hat Y_{1,t-1}^\mathrm{dyn},\hat Y_{1t}^\mathrm{dyn},\hat Y_{1,t+1}^\mathrm{dyn})$ of consecutive scores on the pollution level of day $t$. We do this  by adding the functions\vspace{-1mm}
$$
\mathrm{eff}(\delta_{-1},\delta_0,\delta_1):=\sum_{\ell=-1}^1\delta_\ell\hat\phi_{1\ell}(u),\quad\text{with}\quad  \delta_i=\mathrm{const}\times\pm 1,
\vspace{-1mm}$$ 
to the overall mean curve $\hat\mu(u)$. In Figure~\ref{fig:dynFPC}, we do this with $\delta_i=\pm 1$. For instance, the upper left panel shows $\hat\mu(u)+\mathrm{eff}(-1,-1,-1)$, corresponding to the impact of three consecutive small dynamic FPC scores. The result is a negative shift of the mean curve.  If two small scores are followed by a large one (second panel from the left in top row), then the {\tt PM10} level increases as $u$ approaches 1. Since a large value of $\hat Y_{1,t+1}^\mathrm{dyn}$ implies a large average concentration of $\sqrt{\tt{PM10}}$ on day $t+1$, and since the pollution curves are highly correlated at the transition from day $t$ to day $t+1$, this should indeed be reflected by a higher value of $\sqrt{{\tt PM10}}$ towards the end of day~$t$. Similar interpretations can be given for the other panels in Figure~\ref{fig:dynFPC}. 

It is interesting to observe that, in this example, the first dynamic FPC seems to take over the roles of the first two static FPCs. The second static FPC (see Figure~\ref{fig:pm_pca}) indeed can be interpreted as an intraday trend effect; if the second static score of day $t$ is large (small), then $X_t(u)$ is increasing (decreasing) over $u\in[0,1]$. 
Since we are working with sequentially dependent data, we can get information about such a trend from future and past observations, too. Hence, roughly speaking, we have
$$
\sum_{\ell=-1}^1\hat Y_{1,t+\ell}^\mathrm{dyn}\hat \phi_{1\ell}(u)\approx \sum_{m=1}^2\hat Y_{mt}^{\mathrm{stat}}\hat v_m(u).
$$
This is exemplified in 
Figure~\ref{fig:comp} of Section~\ref{se:intro}, which shows the ten consecutive curves $x_{71}(u)-\hat\mu(u),\ldots,x_{80}(u)-\hat\mu(u)$ (left panel) and compares them to the single-term static (middle panel) and the single-term dynamic Karhunen-Lo\`eve expansions (right panel).
\begin{figure}[ht]
\centering
\includegraphics[width=12.5cm]{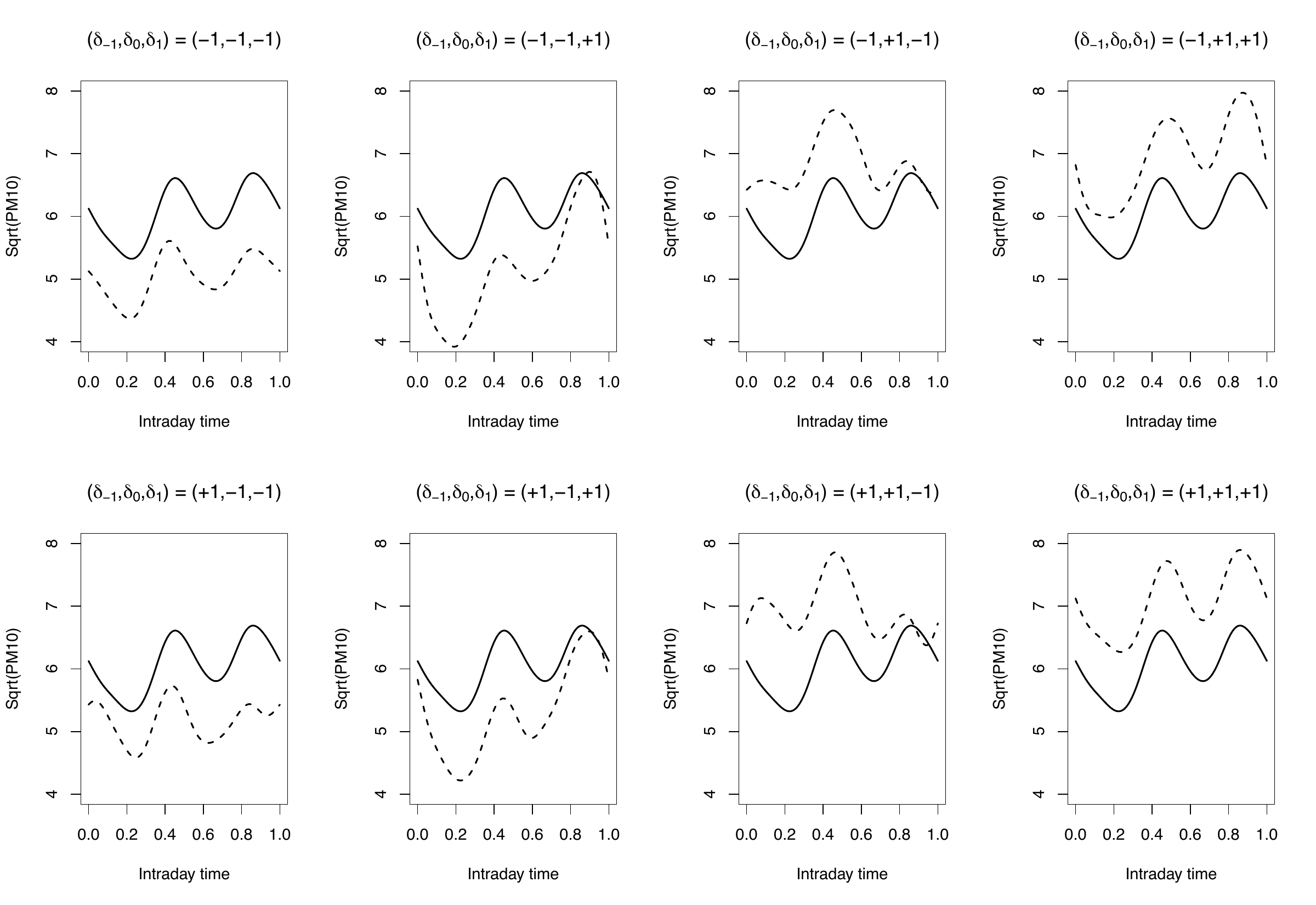}
\caption{Mean curves $\hat\mu(u)$ (solid line) and $\hat\mu(u)+\mathrm{eff}(\delta_{-1},\delta_0,\delta_1)$, with $\delta_i=\pm 1$ (dashed).}
\label{fig:dynFPC}
\end{figure}

\section{Simulation study}\label{se:simul}

In this simulation study, we compare the performance of dynamic FPCA with that of static FPCA for a variety of data-generating processes. For each simulated  functional time series $(X_t)$, where $X_t=X_t(u)$, $u\in[0,1]$, we compute the static and dynamic scores, and recover the approximating series $(\hat X_t^{\mathrm{stat}}(p))$ and $(\hat X_t^{\mathrm{dyn}}(p))$ that result from  the static and dynamic Karhunen-Lo\`eve expansions, respectively, of order $p$.  The performances of these approximations  are  measured in terms of the corresponding normalized mean squared errors  (NMSE)
$$
\sum_{t=1}^{n}\|X_t-\hat X_t^{\mathrm{stat}}(p)\|^2\Big\slash \sum_{t=1}^{n}\|X_t\|^2
\quad
\text{and}
\quad
\sum_{t=1}^{n}\|X_t-\hat X_t^{\mathrm{dyn}}(p)\|^2\Big\slash \sum_{t=1}^{n}\|X_t\|^2.
$$
The smaller these quantities, the better  the approximation. 

Computations were 
 implemented in  {\tt R}, along with the {\tt fda} package. The data were simulated according to a functional AR(1) model $X_{n+1}=\Psi(X_n)+\varepsilon_{n+1}$. In practice,  this simulation has to be performed  in finite dimension $d$, say. To this end, let $(v_i)$, $i\in\mathbb{N}$
  be the Fourier basis functions on $[0,1]$: for large $d$, due to the   linearity of~$\Psi$, 
\begin{align*}
&\langle X_{n+1},v_j\rangle =\langle \Psi(X_n),v_j\rangle +\langle \varepsilon_{n+1},v_j\rangle\\
&\quad= \langle \Psi\big(\sum_{i=1}^\infty\langle X_n,v_i\rangle v_i\big),v_j\rangle +\langle \varepsilon_{n+1},v_j\rangle\approx \sum_{i=1}^d\langle X_n,v_i\rangle \langle \Psi(v_i),v_j\rangle +\langle \varepsilon_{n+1},v_j\rangle.
\end{align*}
Hence, letting $\bX_n=(\langle X_n,v_1\rangle,\ldots,\langle X_n,v_d\rangle)^{\prime}$ and $\boldsymbol{\varepsilon}_n=(\langle \varepsilon_{n+1},v_1\rangle,\ldots,\langle \varepsilon_{n+1},v_d\rangle)^{\prime}$,   the first $d$ Fourier coefficients of $X_n$ approximately  satisfy the VAR(1) equa\-tion\linebreak $
\bX_{n+1}=\mathfrak{P}\bX_n+\boldsymbol{\varepsilon}_n,
$ 
where $\mathfrak{P}=(\langle \Psi(v_i),v_j\rangle\colon 1\leq i,j\leq d)$.
Based on this observation, we used a VAR(1) model for generating the first $d$ Fourier coefficients of the process~$(X_n)$. To obtain $\mathfrak{P}$, we generate a matrix $G=(G_{ij}\colon 1 \leq i,j \leq d)$, where  the $G_{ij}$'s are mutually independent $N(0,\psi_{ij})$, and then set $\mathfrak{P}:=\kappa{G}/{\|G\|}$. Different choices of $\psi_{ij}$ are considered.
Since $\Psi$ is bounded, we have $\mathfrak{P}_{ij}\to 0$ as  $i,j\to\infty$.  For the operators~$\Psi_1$, $\Psi_2$ and $\Psi_3$, we used $\psi_{ij}=(i^2+j^2)^{-1/2}$,   $\psi_{ij}=(i^{2/2}+j^{3/2})^{-1}$, and~$\psi_{ij}=e^{-(i+j)}$, respectively. 
For $d$ and  $\kappa$, we considered the values $d=15,31,51,101$ and  $\kappa=0.1, 0.3, 0.6, 0.9$. The noise $(\varepsilon_t)$ is chosen as independent Gaussian and obtained as a linear combination of the functions $(v_i\colon 1 \leq i \leq d)$ with independent zero-mean normal coefficients $(C_i\colon 1 \leq i \leq d)$, such that $\mathrm{Var}(C_i)=\exp((i-1)/10)$.
With this approach, we generate $n=400$ observations. We then follow the methodology described in Section~\ref{se:pract} and use the Barlett kernel in \eqref{e:bartlett_estimator} for estimation of the spectral density operator. The tuning parameter $q$ is set equal to $\sqrt{n}=20$.
A more sophisticated calibration probably  can lead to even better results, but we also observed that moderate variations of~$q$ do not fundamentally change our findings. The numerical integration for obtaining~$\hat\phi_{mk}$ is performed on the basis of $1000$ equidistant integration points. In \eqref{e:est_scores_equation} we chose~$L = \min(L^{\prime},60)$, where $L^{\prime}=\mathrm{argmin}_{j\geq 0}\sum_{-j\leq \ell\leq j}\|\hat\phi_{m\ell}\|^2\geq 0.99$. The limitation $L\leq 60$ is imposed to keep computation times moderate. Usually, convergence is relatively fast.

For each choice of $d$ and $\kappa$, the experiment as described above is repeated $200$ times. The mean and standard deviation of NMSE in different settings and with values $p=1,2,3,6$ are reported in Table~\ref{tab:simulation}.
 Results do not vary much among setups with $d \geq 31$, and thus in Table~\ref{tab:simulation} we only present the cases $d=15$ and $d=101$.

We see that, in basically all settings, dynamic FPCA significantly outperforms static FPCA in terms of NMSE.  As one can expect, the difference becomes more striking with increasing  dependence coefficient $\kappa$. It is also interesting to observe that the variations of NMSE among the 200 replications is systematically smaller for the dynamic procedure.

Finally, it should be noted that, in contrast to the static PCA, the empirical version of our procedure is not ``exact'', but is subject to small approximation errors. These approximation errors can stem from numerical integration (which is required in the calculation of $\hat\phi_{mk}$) and are also due to the truncation of the filters at some finite lag $L$ (see Section~\ref{se:pract}). Such little deviations do not matter in practice if a component explains a significant proportion of variance. If, however, the additional contribution of the higher-order component is very small, then it can happen that it doesn't compensate a possible approximation error.  This becomes visible in the setting $\Psi_3$ with 3 or 6 components, where for some constellations the NMSE for dynamic components is slightly larger than for the static ones.

\begin{sidewaystable}
\begin{center}
{ \footnotesize \begin{tabular}{lll|ll|ll|ll|ll}
& & & \multicolumn{2}{c|}{$1$ component} & \multicolumn{2}{c|}{$2$ components} & \multicolumn{2}{c|}{$3$ components} & \multicolumn{2}{c}{$6$ components}\\
& $d$ &$\kappa$ & static & dynamic & static & dynamic & static & dynamic & static & dynamic\\

\hline
\multirow{8}{*}{$\Psi_1$}
& \multirow{4}{*}{15}
& 0.1 &  {\bf 0.697} (0.16) &  {\bf 0.637} (0.13) &  {\bf 0.546} (0.15) &  {\bf 0.447} (0.10) &  {\bf 0.443} (0.12) &  {\bf 0.325} (0.08) &  {\bf 0.256} (0.08) &  {\bf 0.138} (0.05)\\
& & 0.3 &  {\bf 0.696} (0.16) &  {\bf 0.621} (0.14) &  {\bf 0.542} (0.15) &  {\bf 0.434} (0.11) &  {\bf 0.440} (0.13) &  {\bf 0.314} (0.08) &  {\bf 0.253} (0.08) &  {\bf 0.132} (0.05)\\ 
& & 0.6 &  {\bf 0.687} (0.32) &  {\bf 0.571} (0.23) &  {\bf 0.526} (0.25) &  {\bf 0.392} (0.15) &  {\bf 0.423} (0.20) &  {\bf 0.283} (0.11) &  {\bf 0.240} (0.11) &  {\bf 0.119} (0.06)\\ 
& & 0.9 &  {\bf 0.648} (0.76) &  {\bf 0.479} (0.47) &  {\bf 0.481} (0.56) &  {\bf 0.322} (0.29) &  {\bf 0.377} (0.43) &  {\bf 0.229} (0.20) &  {\bf 0.209} (0.22) &  {\bf 0.096} (0.09)\\
& \multirow{4}{*}{101}
& 0.1 &  {\bf 0.805} (0.12) &  {\bf 0.740} (0.08) &  {\bf 0.708} (0.11) &  {\bf 0.587} (0.08) &  {\bf 0.642} (0.12) &  {\bf 0.478} (0.07) &  {\bf 0.519} (0.08) &  {\bf 0.274} (0.05)\\ 
& & 0.3 &  {\bf 0.802} (0.13) &  {\bf 0.729} (0.11) &  {\bf 0.704} (0.12) &  {\bf 0.577} (0.09) &  {\bf 0.637} (0.11) &  {\bf 0.469} (0.08) &  {\bf 0.515} (0.10) &  {\bf 0.269} (0.05)\\ 
& & 0.6 &  {\bf 0.792} (0.22) &  {\bf 0.690} (0.18) &  {\bf 0.689} (0.19) &  {\bf 0.545} (0.12) &  {\bf 0.619} (0.16) &  {\bf 0.441} (0.10) &  {\bf 0.495} (0.13) &  {\bf 0.252} (0.07)\\ 
& & 0.9 &  {\bf 0.755} (0.66) &  {\bf 0.616} (0.45) &  {\bf 0.640} (0.50) &  {\bf 0.479} (0.31) &  {\bf 0.568} (0.40) &  {\bf 0.387} (0.23) &  {\bf 0.446} (0.34) &  {\bf 0.220} (0.15)\\ 
\hline
\multirow{8}{*}{$\Psi_2$}
& \multirow{4}{*}{15}
& 0.1 &  {\bf 0.524} (0.20) &  {\bf 0.491} (0.17) &  {\bf 0.355} (0.14) &  {\bf 0.306} (0.11) &  {\bf 0.263} (0.10) &  {\bf 0.208} (0.08) &  {\bf 0.129} (0.05) &  {\bf 0.082} (0.03)\\ 
& & 0.3 &  {\bf 0.522} (0.21) &  {\bf 0.473} (0.18) &  {\bf 0.351} (0.16) &  {\bf 0.294} (0.12) &  {\bf 0.259} (0.12) &  {\bf 0.200} (0.08) &  {\bf 0.126} (0.06) &  {\bf 0.078} (0.04)\\ 
& & 0.6 &  {\bf 0.507} (0.49) &  {\bf 0.413} (0.29) &  {\bf 0.331} (0.29) &  {\bf 0.255} (0.15) &  {\bf 0.240} (0.19) &  {\bf 0.174} (0.10) &  {\bf 0.114} (0.08) &  {\bf 0.068} (0.05)\\ 
& & 0.9 &  {\bf 0.458} (1.15) &  {\bf 0.310} (0.59) &  {\bf 0.272} (0.64) &  {\bf 0.187} (0.32) &  {\bf 0.193} (0.41) &  {\bf 0.130} (0.21) &  {\bf 0.088} (0.17) &  {\bf 0.052} (0.09)\\ 
& \multirow{4}{*}{101}
& 0.1 &  {\bf 0.585} (0.19) &  {\bf 0.549} (0.17) &  {\bf 0.436} (0.15) &  {\bf 0.378} (0.11) &  {\bf 0.356} (0.13) &  {\bf 0.282} (0.10) &  {\bf 0.240} (0.08) &  {\bf 0.146} (0.05)\\ 
& & 0.3 &  {\bf 0.581} (0.21) &  {\bf 0.530} (0.18) &  {\bf 0.436} (0.12) &  {\bf 0.369} (0.11) &  {\bf 0.350} (0.13) &  {\bf 0.274} (0.09) &  {\bf 0.234} (0.10) &  {\bf 0.141} (0.06)\\ 
& & 0.6 &  {\bf 0.564} (0.46) &  {\bf 0.469} (0.27) &  {\bf 0.405} (0.33) &  {\bf 0.321} (0.18) &  {\bf 0.323} (0.21) &  {\bf 0.242} (0.13) &  {\bf 0.212} (0.12) &  {\bf 0.125} (0.07)\\ 
& & 0.9 &  {\bf 0.495} (1.06) &  {\bf 0.362} (0.59) &  {\bf 0.345} (0.68) &  {\bf 0.250} (0.39) &  {\bf 0.251} (0.58) &  {\bf 0.180} (0.34) &  {\bf 0.168} (0.26) &  {\bf 0.097} (0.14)\\ 
\hline
\multirow{8}{*}{$\Psi_3$}
& \multirow{4}{*}{15}
& 0.1 &  {\bf 0.367} (0.20) &  {\bf 0.344} (0.18) &  {\bf 0.134} (0.08) &  {\bf 0.127} (0.07) &  {\bf 0.049} (0.03) &  {\bf 0.054} (0.04) &  {\bf 0.002} (0.00) &  {\bf 0.017} (0.03)\\ 
& & 0.3 &  {\bf 0.362} (0.24) &  {\bf 0.322} (0.17) &  {\bf 0.129} (0.09) &  {\bf 0.119} (0.07) &  {\bf 0.048} (0.03) &  {\bf 0.050} (0.04) &  {\bf 0.002} (0.00) &  {\bf 0.015} (0.03)\\ 
& & 0.6 &  {\bf 0.334} (0.55) &  {\bf 0.253} (0.24) &  {\bf 0.113} (0.16) &  {\bf 0.097} (0.09) &  {\bf 0.041} (0.05) &  {\bf 0.040} (0.04) &  {\bf 0.002} (0.00) &  {\bf 0.011} (0.02)\\ 
& & 0.9 &  {\bf 0.236} (1.12) &  {\bf 0.146} (0.43) &  {\bf 0.074} (0.28) &  {\bf 0.061} (0.16) &  {\bf 0.025} (0.08) &  {\bf 0.027} (0.07) &  {\bf 0.001} (0.00) &  {\bf 0.008} (0.04)\\ 
& \multirow{4}{*}{101}
& 0.1 &  {\bf 0.366} (0.19) &  {\bf 0.344} (0.17) &  {\bf 0.134} (0.08) &  {\bf 0.127} (0.07) &  {\bf 0.049} (0.03) &  {\bf 0.054} (0.04) &  {\bf 0.002} (0.00) &  {\bf 0.017} (0.03)\\ 
& & 0.3 &  {\bf 0.363} (0.25) &  {\bf 0.322} (0.18) &  {\bf 0.131} (0.10) &  {\bf 0.120} (0.07) &  {\bf 0.047} (0.03) &  {\bf 0.050} (0.04) &  {\bf 0.002} (0.00) &  {\bf 0.015} (0.03)\\ 
& & 0.6 &  {\bf 0.325} (0.52) &  {\bf 0.251} (0.24) &  {\bf 0.113} (0.16) &  {\bf 0.098} (0.09) &  {\bf 0.040} (0.05) &  {\bf 0.040} (0.04) &  {\bf 0.002} (0.00) &  {\bf 0.011} (0.02)\\ 
& & 0.9 &  {\bf 0.235} (1.05) &  {\bf 0.149} (0.43) &  {\bf 0.074} (0.28) &  {\bf 0.061} (0.16) &  {\bf 0.025} (0.09) &  {\bf 0.026} (0.07) &  {\bf 0.001} (0.00) &  {\bf 0.008} (0.04)\\ 
\end{tabular}}
\end{center}
\label{tab:simulation}
\caption{
Results of the simulations of Section~\ref{se:simul}. Bold numbers represent the mean of NMSE for dynamic and static procedures resulting from 200 simulation runs. The numbers in brackets show standard deviations multiplied by a factor 10. The values $\kappa$ give the size of $\|\Psi_i\|_\mathcal{L}$, $i=1,2,3$.
We consider dimensions of the underlying models $d=15$ and $d=101$.}
\end{sidewaystable}

\section{Conclusion}\label{se:concl}

Functional principal component analysis is taking a leading role in the functional data literature. As an extremely effective tool for dimension reduction, it is useful for empirical data analysis as well as for many FDA-related methods, like functional linear models. A frequent situation in practice is that functional data are observed sequentially over time and exhibit serial dependence. This happens, for instance, when observations stem from a continuous-time process which is segmented into smaller units, e.g., days. In such cases, classical static FPCA still may be useful, but, in contrast to the i.i.d.\ setup, it does not lead to an optimal dimension-reduction technique. 

In this paper, we   propose a {\em dynamic} version of FPCA which takes advantage of the  potential serial dependencies in the functional observations. In the special case of uncorrelated data, the dynamic FPC methodology reduces to the usual static one. But, in the presence of serial dependence, static FPCA is (quite significantly, if serial dependence is strong) outperformed. 

This paper also provides (i) guidelines for practical implementation, (ii) a toy example with {\tt PM10} air  pollution data, and (iii) a simulation study. Our empirical application brings empirical evidence that dynamic FPCs have a clear edge over static FPCs in terms of their ability to represent dependent functional data in small dimension.  In the appendices, our results are cast into a rigorous mathematical framework, and we show that the proposed estimators of dynamic FPC scores are  consistent.  


\appendix


\section{General methodology and proofs}\label{app:pract}

In this subsection, we give a mathematically rigorous description of the methodology introduced in Section~\ref{se:notat}. We adopt a more general framework which can be specialized to the functional setup of Section~\ref{se:notat}. Throughout, $H$ denotes some (complex) separable Hilbert space equipped with norm $\|\cdot\|$ and inner product $\langle \cdot,\cdot\rangle$. 
We work in complex spaces, since our theory is based on a frequency domain analysis. Nevertheless, all our functional time series observations~$X_t$ are assumed to be real-valued functions.

\subsection{Fourier series in Hilbert spaces.}\label{se:A1}
For $p\geq 1$, consider the space $L_H^p([-\pi,\pi])$, that is,  the space of measurable mappings $x:[-\pi,\pi]\to H$ such that  $\int_{-\pi}^\pi \|x(\theta)\|^pd\theta<~\infty$.
Then, $\|x\|_p=(\frac{1}{2\pi}\int_{-\pi}^\pi\|x(\theta)\|^pd\theta)^{1/p}$ defines a norm. Equipped with this norm,  $L_H^p([-\pi,\pi])$ is a Banach space, and for
$p=2$, a Hilbert space with inner product
$$
(x,y):=\frac{1}{2\pi}\int_{-\pi}^\pi \langle x(\theta),y(\theta)\rangle d\theta.
$$
One can show (see e.g.\ \cite[Lemma 1.4]{bosq:2000}) that, for any $x\in L_H^1([-\pi,\pi])$, there exists a unique element $I(x)\in H$ which satisfies 
\begin{equation}\label{e:defint}
\int_{-\pi}^\pi \langle x(\theta),v\rangle d\theta=\langle I(x),v\rangle\quad \forall v\in H.
\end{equation}
We define $\int_{-\pi}^\pi x(\theta)d\theta:=I(x)$. 

For $x\in L_H^2([-\pi,\pi])$, define the {\it $k$-th Fourier coefficient} as 
\begin{equation}\label{eq:f-k}
f_k:=\frac{1}{2\pi}\int_{-\pi}^\pi x(\theta)e^{-\ii k\theta}d\theta,\quad k\in\mathbb{Z}.
\end{equation}
Below, we write $e_k$ for the function $\theta\mapsto e^{\ii k\theta}$, $\theta\in[-\pi,\pi]$. 
\begin{Proposition}\label{pr:Hfourier} Suppose $x\in L_H^2([-\pi,\pi])$ and 
define $f_k$ by equation~\eqref{eq:f-k}. Then, the sequence $S_n:=\sum_{k=-n}^n f_k e_k$ has a mean square limit in $L_H^2([-\pi,\pi])$.
If we denote the limit by $S$, then
$x(\theta)=S(\theta)$ for almost all $\theta$.
\end{Proposition}

Let us turn to the Fourier expansion of eigenfunctions $\varphi_m(\theta)$ used in the definition of the dynamic DPFCs. 
Eigenvectors are scaled to unit length: $\|\varphi_m(\theta)\|^2=1$. In order for $\varphi_m$ to belong to $L_H^2([-\pi,\pi])$, we additionally need measurability. Measurability cannot be taken for granted. This comes from the fact that $\|z\varphi_m(\theta)\|^2=1$ for all $z$ on the complex unit circle. In principle we could choose the ``signs'' $z=z(\theta)$  in an extremely erratic way, such that $\varphi_m(\theta)$ is no longer measurable. To exclude such pathological choices, we tacitly impose in the sequel that
versions of $\varphi_m(\theta)$ have been chosen in a ``smooth enough way'', to be  measurable.

Now we can expand the eigenfunctions $\varphi_m(\theta)$ in a Fourier series in the sense explained above:
$$
\varphi_m=\sum_{\ell\in\mathbb{Z}}\phi_{m\ell}e_\ell
\quad\text{with}\quad \phi_{m\ell}=\frac{1}{2\pi}\int_{-\pi}^\pi\varphi_m(s)e^{-\ii \ell s}ds.
$$
The coefficients $\phi_{m\ell}$ thus defined yield the definition \eqref{eq:dfpcs} of  dynamic FPCs. In the special case $H=L^2([0,1])$, $\phi_{m\ell}=\phi_{m\ell}(u)$ satisfies by \eqref{e:defint}
\begin{align*}\int_0^1\phi_{m\ell}(u)v(u)du&=\frac{1}{2\pi}\int_{-\pi}^\pi\int_0^1\varphi_m(u|s)v(u)du e^{-\ii \ell s}ds\\
&=\int_0^1\left(\frac{1}{2\pi}\int_{-\pi}^\pi\varphi_m(u|s) e^{-\ii \ell s}ds\right)v(u)du
\quad\forall v\in H.
\end{align*}
This implies that $\phi_{m\ell}(u)=\frac{1}{2\pi}\int_{-\pi}^\pi\varphi_m(u|s) e^{-\ii \ell s}ds$ for almost all $u\in[0,1]$, which is in line with the definition given in~\eqref{e:phiml}. Furthermore, \eqref{eq:convl2} follows directly from Proposition~\ref{pr:Hfourier}.

\subsection{The spectral density operator}\label{se:sdo}

Assume that the $H$-valued process $(X_t\colon t\in\mathbb{Z})$ is stationary with lag $h$ autocovariance operator $C_h^X$ and spectral density operator\vspace{-1mm}
\begin{equation}\label{conv}
\mathcal{F}^X_\theta:=\frac{1}{2\pi}\sum_{h\in\mathbb{Z}}C_h^X e^{-\ii h\theta}.
\vspace{-1mm}\end{equation}
Let $\mathcal{S}(H,H')$ be the set of Hilbert-Schmidt operators mapping from $H$ to $H'$ (both assumed to be separable Hilbert spaces). When $H=H'$ and when it is clear which space $H$ is meant, we sometimes simply write $\mathcal{S}$. With the Hilbert-Schmidt norm $\|\cdot\|_{\mathcal{S}(H,H')}$ this defines again a separable Hilbert space,  and so does $L_{\mathcal{S}(H,H')}^2([-\pi,\pi])$.
We will impose that the series in \eqref{conv} converges in $L_{\mathcal{S}(H,H)}^2([-\pi,\pi])$: we then  say that~$(X_t)$ possesses a spectral density operator.
\begin{Remark}\label{rem3} It follows that the results of the previous section can be applied. In particular we may deduce that
$C_k^X=\int_{-\pi}^\pi \mathcal{F}_\theta^X e^{\ii k\theta}d\theta$. 
\end{Remark}

A sufficient condition for convergence of (\ref{conv}) in $L^2_{\mathcal{S}(H,H)}([-\pi,\pi])$ is assumption~\eqref{e:abssymcov}. Then, it can be easily shown that the operator $\mathcal{F}^X_\theta$ is self-adjoint, non-negative definite and Hilbert-Schmidt. Below, we introduce a weak dependence assumption established in \cite{hormann:kokoszka:2010}, from which we can derive a sufficient condition for~\eqref{e:abssymcov}.

\begin{Definition}[$L^p$--$m$--approximability]\label{def:mapr}
A random $H$--valued sequence $(X_n\colon n\in\mathbb{Z})$ is called $L^p$--$m$--approximable if it can be represented as 
 $
X_n = f(\delta_n,\delta_{n-1},\delta_{n-2},...)
$,  
where the $\delta_i$'s are i.i.d.\ elements taking values in some measurable space $S$ and $f$ is a measurable function $f : S^\infty \rightarrow H$. Moreover, if $\delta'_1,\delta'_2,...$ are independent copies of~$\delta_1,\delta_2,...$ defined on the same measurable space $S$, then, for
\begin{align*}
X_n^{(m)}:= f(\delta_n,\delta_{n-1},\delta_{n-2},...,\delta_{n-m+1},\delta'_{n-m},\delta'_{n-m-1},...),
\end{align*}
we have
\begin{align}\label{lpm}
\sum_{m=1}^\infty (E\|X_m - X_m^{(m)}\|^p)^{1/p} < \infty.
\end{align}
\end{Definition}
H\"ormann and Kokoszka~\cite{hormann:kokoszka:2010} show that this notion is widely applicable to linear and non-linear functional time series. One of its main advantages is that it is a purely moment-based dependence measure that can be easily verified in many special cases.

\begin{Proposition}\label{pr:summab}
Assume that $(X_t)$ is $L^2$--$m$--approximable. Then \eqref{e:abssymcov} holds and  the operators $\mathcal{F}_\theta^X$, $\theta\in[-\pi,\pi]$, are trace-class.
\end{Proposition}
Instead of Assumption \eqref{e:abssymcov}, Panaretos and Tavakoli~\cite{panaretos:tavakoli:2012} impose for the definition of a spectral density operator summability of $C_h^X$ in Schatten 1-norm, that is,~$\sum_{h\in\mathbb{Z}}\|C_h^X\|_\mathcal{T}<\infty$. Under such  slightly more stringent assumption, it immediately follows that the resulting spectral operator is trace-class. The verification of convergence may, however, be a bit delicate. At least, we could not find a simple criterion as in Proposition~\ref{pr:summab}.


\begin{Proposition}\label{pr:conteigen} Let $\mathcal{F}_\theta^X$ be the spectral density operator of a stationary sequence~$(X_t)$ for which the summability condition \eqref{e:abssymcov} holds.
Let $\lambda_1(\theta)\geq\lambda_2(\theta)\geq\cdots$ denote its eigenvalues and $\varphi_m(\theta)$ be the corresponding eigenfunctions. Then,
(a) the functions $\theta\mapsto \lambda_m(\theta)$ are continuous; (b)
if we strengthen \eqref{e:abssymcov} into the more stringent condition 
$
\sum_{h\in\mathbb{Z}}|h|\|C_h^X\|_\mathcal{S}<~\infty,
$
the $\lambda_m(\theta)$'s are Lipschitz-continuous functions of $\theta$;
(c) assuming that $(X_t)$ is real-valued, for each~$\theta \in [-\pi,\pi]$, $\lambda_m(\theta)=\lambda_m(-\theta)$ and $\varphi_m(\theta)=\overline{\varphi_m(-\theta)}$.
\end{Proposition}

Let $\overline{x}$ be the conjugate element of $x$, i.e.\ $\langle x,z\rangle= \langle z,\overline{x}\rangle$ for all $z\in H$. Then $x$ is real-valued iff $x=\overline{x}$.

\begin{Remark}\label{rem:2}
Since $\varphi_m(\theta)$ is Hermitian, it immediately follows that~$\phi_{m\ell}=\overline{\phi_{m\ell}},$
implying that the dynamic FPCs are real if the process $(X_t)$ is. 
\end{Remark}

\subsection{Functional filters}\label{ss:ffilters}

Computation of dynamic FPCs requires applying time-invariant {\it functional filters} to the process $(X_t)$. Let $\boldsymbol{\Psi}=(\Psi_{k}\colon k\in\mathbb{Z})$ be a sequence of linear operators  mapping  the separable Hilbert space  $H$ to the separable Hilbert space $H^{\prime}$.  Let $B$ be the backshift or lag operator, defined by~$B^k X_t:=X_{t-k}$, $k\in\mathbb{Z}$. Then the functional filter~$\Psi(B):=\sum_{k\in\mathbb{Z}}\Psi_kB^k$, when applied to the sequence $(X_t)$, produces an output series $(Y_t)$ in $H^{\prime}$ via
\begin{equation}\label{e:Y}
Y_t=\Psi(B)X_t=\sum_{k\in\mathbb{Z}}\Psi_k(X_{t-k}).
\end{equation} 
Call $\boldsymbol{\Psi}$ the sequence of filter coefficients, and,
in the style of the scalar or vector time series terminology,   call 
\begin{equation}\label{e:Psi}
\Psi_\theta=\Psi(e^{-\ii\theta})=\sum_{k\in\mathbb{Z}}\Psi_ke^{-\ii k\theta}
\end{equation}
the {\em frequency response function} of the filter $\Psi(B)$. Of course, series \eqref{e:Y} and \eqref{e:Psi} only have a meaning if they converge in an appropriate sense. Below we use the following technical lemma.

\begin{Proposition}\label{pr:sd_filter0} 
Suppose that $(X_t)$ is a stationary sequence in $L_H^2$ and possesses a spectral density operator satisfying \ $\sup_\theta \tr(\mathcal{F}_\theta^X)<\infty$.
Consider a filter $(\Psi_k)$ such that $\Psi_\theta$ converges in $L^2_{\mathcal{S}(H,H^{\prime})}([-\pi,\pi])$, and suppose that $\sup_{\theta}\|\Psi_\theta\|_{\mathcal{S}(H,H^{\prime})}<\infty$. Then,
\begin{enumerate}[(i)]
\item  the series $Y_t:=\sum_{k\in\mathbb{Z}}\Psi_k(X_{t-k})$ converges in $L^2_{H^{\prime}}$;
\item $(Y_t)$ possesses the  spectral density operator  $\mathcal{F}_\theta^{Y}=\Psi_\theta\mathcal{F}_\theta^X(\Psi_\theta)^*$;
\item $\sup_\theta\tr(\mathcal{F}_\theta^Y)<\infty$.
\end{enumerate}
\end{Proposition}

\noindent In particular, the last proposition allows for iterative applications. If $\sup_\theta \tr(\mathcal{F}_\theta^X)<~\infty$ and $\Psi_\theta$ satisfies the above properties, then  analogue results   apply to the output $Y_t$. This is what we are using in the proofs of Theorems~1 and 2.

\subsection{Proofs for Section~\ref{se:metho}}

To start with, observe that Propositions~\ref{pr:sd_filter} and~\ref{pr:secondorder} directly follow  from Proposition~\ref{pr:sd_filter0}. Part~(a) of Proposition~\ref{pr:elementary} also has been established in the previous Section (see Remark~\ref{rem:2}), and part (b) is immediate. Thus, we can proceed to the proof of  Theorems~1 and~2. 


\begin{proof}[Proof of Theorems~\ref{th:inversion} and~\ref{th:optimality}]
Assume we have filter coefficients $\boldsymbol{\Psi}=(\Psi_k\colon k\in\mathbb{Z})$ and $\boldsymbol{\Upsilon}=(\Upsilon_k\colon k\in\mathbb{Z})$, where $\Psi_k:H\to \mathbb{C}^p$ and $\Upsilon_k:\mathbb{C}^p\to H$ both belong to the class $\mathcal{C}$. If $(X_t)$ and $(Y_t)$ are $H$-valued and $\mathbb{C}^p$-valued processes, respectively, then there exist elements $\psi_{mk}$ and $\upsilon_{mk}$ in $H$ such that
$$
\Psi(B)(X_t)=\sum_{k\in\mathbb{Z}}\left(\langle X_{t-k},\psi_{1k}\rangle,\ldots,\langle X_{t-k},\psi_{pk}\rangle\right)^{\prime}
\vspace{-2mm}$$ 
and
$$
\Upsilon(B)(Y_t)=\sum_{\ell\in\mathbb{Z}}\sum_{m=1}^pY_{t+\ell,m}
\upsilon_{m\ell}.
$$
Hence, the $p$-dimensional reconstruction of $X_t$ in Theorem~\ref{th:optimality} is of the form $$\sum_{m=1}^p\tilde X_{mt}=\Upsilon(B)[\Psi(B)X_t]=:\Upsilon\Psi(B)X_t.$$
Since  $\boldsymbol{\Psi}$ and $\boldsymbol{\Upsilon}$ are required  to belong to $\mathcal{C}$, we conclude from Proposition~\ref{pr:sd_filter0} that the processes $Y_t:=\Psi(B)X_t$ and $\tilde X_t=\Upsilon(B)Y_t$ are mean-square convergent and possess a spectral density operator. 
Letting $\psi_m(\theta)=\sum_{k\in\mathbb{Z}}\psi_{mk}e^{\ii k\theta}$ and $\upsilon_m(\theta)=\sum_{\ell\in\mathbb{Z}}\upsilon_{m\ell}e^{\ii \ell\theta}$, we obtain, for $x\in H$ and $y=(y_1,\ldots,y_m)^{\prime}\in\mathbb{C}^p$, that the frequency response functions $\Psi_\theta$ and $\Upsilon_\theta$ satisfy
$$
\Psi_\theta(x)=\sum_{k\in\mathbb{Z}}\left(\langle x,\psi_{1k}\rangle,\ldots,\langle x,\psi_{pk}\rangle\right)^{\prime}e^{-\ii k\theta}=\left(\langle x,\psi_1(\theta)\rangle,\ldots,\langle x,\psi_p(\theta)\rangle\right)^{\prime}
\vspace{-2mm}$$
and
$$
\Upsilon_\theta(y)=\sum_{\ell\in\mathbb{Z}}\sum_{m=1}^py_m  \upsilon_{m\ell}e^{-\ii \ell\theta}=\sum_{m=1}^py_m \upsilon_m(-\theta).
$$
Consequently,
\begin{equation}\label{BA}
\Upsilon_\theta \Psi_\theta=\sum_{m=1}^p \upsilon_m(-\theta)\otimes \psi_m(\theta).
\end{equation}
Now, using Proposition~\ref{pr:sd_filter0}, it is readily verified that, for $Z_t:=X_t-\Upsilon\Psi(B)X_t$, we obtain the spectral density operator
\begin{equation}
\mathcal{F}^{Z}_\theta
=\left(\tilde{\mathcal{F}}^X_\theta-\Upsilon_\theta \Psi_\theta \tilde{\mathcal{F}}^X_\theta\right)\left(\tilde{\mathcal{F}}^X_\theta- \tilde{\mathcal{F}}^X_\theta \Psi^*_\theta \Upsilon_\theta^* \right),
\end{equation}
where $\tilde{\mathcal{F}}^X_\theta$ is such that $\tilde{\mathcal{F}}^X_\theta\tilde{\mathcal{F}}^X_\theta= \mathcal{F}^X_\theta.$ 

Using Lemma~\ref{le:trace},
\begin{equation}
E\|X_t-\Upsilon\Psi(B)X_t\|^2
=\int_{-\pi}^\pi\tr(\mathcal{F}_\theta^Z)\,d\theta
=\int_{-\pi}^\pi\Big\|\tilde{\mathcal{F}}^X_\theta-\Upsilon_\theta \Psi_\theta \tilde{\mathcal{F}}^X_\theta\Big\|_\mathcal{S}^2d\theta.\label{e:tr}
\end{equation}
Clearly, \eqref{e:tr} is minimized if we minimize the integrand for every fixed $\theta$ under the constraint that $\Upsilon_\theta \Psi_\theta$ is of the form \eqref{BA}. Employing the eigendecompo\-sition $
\mathcal{F}^X_\theta=\sum_{m\geq 1}\lambda_m(\theta)\varphi_m(\theta)\otimes\varphi_m(\theta)$, 
we infer that
$$
\tilde{\mathcal{F}}^X_\theta=\sum_{m\geq 1}\sqrt{\lambda_m(\theta)}\varphi_m(\theta)\otimes\varphi_m(\theta).
\vspace{-2mm}$$
The best approximating operator of rank $p$ to $\tilde{\mathcal{F}}_\theta^X$ is the operator
$$
\tilde{\mathcal{F}}^X_\theta(p)=\sum_{m= 1}^p\sqrt{\lambda_m(\theta)}\varphi_m(\theta)\otimes\varphi_m(\theta),
$$
which  is obtained if we choose $\Upsilon_\theta \Psi_\theta=\sum_{m=1}^p\varphi_m(\theta)\otimes\varphi_m(\theta)$ and hence
$$
\psi_m(\theta)=\varphi_m(\theta)\quad\text{and}\quad \upsilon_m(\theta)=\varphi_m(-\theta).
$$
Consequently, by Proposition~\ref{pr:Hfourier}, we get
$$
\psi_{mk}=\frac{1}{2\pi}\int_{-\pi}^\pi\varphi_m(s)e^{-\ii ks}ds\quad \text{and}\quad
\upsilon_{mk}=\frac{1}{2\pi}\int_{-\pi}^\pi\varphi_m(-s)e^{-\ii ks}ds=\psi_{m,-k}.
$$
With this choice, it is clear that 
$\Upsilon\Psi(B)X_t=\sum_{m=1}^p X_{mt}$ and 
$$E\|X_t-\sum_{m=1}^pX_{mt}\|^2=\int_{-\pi}^\pi\Big\|\tilde{\mathcal{F}}^X_\theta- \tilde{\mathcal{F}}^X_\theta(p)\Big\|_\mathcal{S}^2d\theta=\int_{-\pi}^\pi\sum_{m>p}\lambda_m(\theta)d\theta;$$
the proof of Theorem~\ref{th:optimality} follows.

Turning to Theorem~\ref{th:inversion}, observe that, by the monotone convergence theorem, the last integral tends to zero if~$p\to \infty$, which completes the proof of Theorem~\ref{th:inversion}.
\end{proof}

\section{Large sample properties}\label{app:largesampleproperties}
For the proof of Theorem~\ref{th:dfpc:consistency}, let us  show that $E | Y_{mt} - \hat Y_{mt} |\to 0$ as $n\to\infty$.

Fixing  $L\geq 1$, \vspace{-2mm}
\begin{align}
E | Y_{mt} - \hat Y_{mt} | &\leq E \bigg|\Sum_{j \in \mathbb{Z}} \ip{X_{t-j}}{\phi_{mj}} - \Sum_{j=-L}^L \ip{X_{t-j}}{\hat\phi_{mj}} \bigg|\nonumber\\
&\leq E\bigg|\Sum_{j=-L}^{L} \ip{X_{t-j}}{\phi_{mj} - \hat\phi_{mj}}\bigg| + E\bigg|\Sum_{|j| > L} \ip{X_{t-j}}{\phi_{mj}} \bigg|,\label{proof:split}
\end{align}
and the result follows if each summand in \eqref{proof:split} converges to zero, which we prove in the two subsequent lemmas. For notational convenience, we often suppress the dependence on the sample size $n$; all limits below, however, are taken as $n\to\infty$.
\begin{lemma}\label{lem:sum2}
If $L=L(n)\to\infty$ sufficiently slowly, then, under Assumptions $B.1$--$B.3$, we have that
 $
\Big|\Sum_{|j| \leq L} \ip{X_{k-j}}{\phi_{mj} - \hat \phi_{mj}}\Big| =o_P(1).
$ 
\end{lemma}
\begin{proof}
The triangle  and  Cauchy-Schwarz inequalities yield
\begin{align*}
\Big|\Sum_{|j| \leq L} \ip{X_{k-j}}{\phi_{mj} - \hat \phi_{mj}}  \Big| &\leq \Sum_{j=-L}^L \| X_{k-j} \| \| \phi_{mj} - \hat \phi_{mj} \| \\
&\leq \max_{j\in \mathbb{Z}} \| \phi_{mj} - \hat \phi_{mj} \| \Sum_{j=-L}^L \| X_{k-j} \|.
\end{align*}
Let $\hat c_m(\theta):=\langle \phi_m(\theta),\hat\phi_m(\theta)\rangle/|\langle \phi_m(\theta),\hat\phi_m(\theta)\rangle|$. Jensen's inequality and the triangular inequality imply that, for any~$j \in \mathbb{Z}$,
\begin{align*}
2 \pi \| \phi_{mj} - \hat \phi_{mj} \| &= \Big\|\int_{-\pi}^\pi (\varphi_{m}(\theta) - \hat \varphi_{m}(\theta)) e^{\ii j\theta} d\theta \Big\|          
\leq \int_{-\pi}^\pi \|\varphi_{m}(\theta) - \hat \varphi_{m}(\theta)\| d\theta\\
&\leq \int_{-\pi}^\pi \|\varphi_{m}(\theta) -  \hat c_m(\theta)\hat \varphi_{m}(\theta)\| d\theta+
\int_{-\pi}^\pi|1-\hat c_m(\theta)|d\theta\\
&=:Q_1 +Q_2 .
\end{align*}
By Lemma~3.2 in \cite{hormann:kokoszka:2010}, we have
$$
Q_1 \leq \int_{-\pi}^\pi \frac{8}{|\alpha_m(\theta)|^2} \| \mathcal{F}_\theta^X - \hat{ \mathcal{F}}_\theta^X\|_\cS \wedge 2\ d\theta.
$$

By Assumption~B.2, $\alpha_m(\theta)$ has only finitely many zeros, $\theta_1,\ldots,\theta_K$, say. Let  $\delta_\varepsilon (\theta):=[\theta - \varepsilon,\theta + \varepsilon]$ and $A(m,\varepsilon):= \bigcup_{i=1}^K \delta_\varepsilon(\theta_i)$. By definition, the Lebesgue measure of this set is $|A(m,\varepsilon)| \leq 2K\varepsilon$. Define $M_\varepsilon$ such that
$$M_\varepsilon^{-1} = \min\{ \alpha_m(\theta)\ |\ \theta \in [-\pi,\pi] \backslash A(m,\varepsilon) \}.$$
 By continuity of $\alpha_m(\theta)$ (see Proposition~\ref{pr:conteigen}), we have $M_\varepsilon < \infty$, and thus
\begin{align*}
\int_{-\pi}^\pi \frac{8}{|\alpha_m(\theta)|^2} \| \mathcal{F}_\theta^X - \hat{ \mathcal{F}}_\theta^X\| \wedge 2 d\theta \leq 4K\varepsilon + 8M_\varepsilon^2\int_{-\pi}^\pi \| \mathcal{F}_\theta^X - \hat{\mathcal{F}}_\theta^X\| d\theta =: B_{n, \varepsilon}.
\end{align*}
By Assumption B.1, there exists a sequence $\varepsilon_n \rightarrow 0$ such that $B_{n,\varepsilon_n} \rightarrow 0$ in probability, which entails  $Q_1 = o_P(1)$. Note that this also  implies  
\begin{equation}\label{Q1cons}
\int_{-\pi}^\pi\big|\langle \varphi_m(\theta),v\rangle -\hat c_m(\theta)\langle \hat\varphi_m(\theta),v\rangle\big|d\theta =o_P(1).
\end{equation}

Turning to $Q_2$, suppose that $Q_2$ is not $o_P(1)$ Then, there exists  $\varepsilon>0$ and  $\delta>0$ such that ,for infinitely many $n$,   $P(Q_2 \geq \varepsilon)\geq \delta$. Set 
$$F=F_n:=\left\{\theta\in[-\pi,\pi]\colon|\hat c_m(\theta)-1|\geq \frac{\varepsilon}{4\pi}\right\}.$$
 One can easily show that, on the set $\{Q_2 \geq \varepsilon\}$, we have $\lambda(F)>\varepsilon/4$. Clearly, $|\hat c_m(\theta)-1|\geq {\varepsilon}/{4\pi}$ implies that $\hat c_m(\theta)=e^{\ii z(\theta)}$ with $z(\theta)\in[-\pi/2,-\varepsilon^{\prime}]\cup[\varepsilon^{\prime},\pi/2]$, for some small enough $\varepsilon^{\prime}$. Then the left-hand side in \eqref{Q1cons} is bounded from below by
\begin{align}
&\int_F\big|\langle \varphi_m(\theta),v\rangle -\hat c_m(\theta)\langle \hat\varphi_m(\theta),v\rangle\big|d\theta\nonumber\\
&\quad=\int_F\left(
\langle (\varphi_m(\theta),v\rangle-\cos(z(\theta)\langle \hat\varphi_m(\theta),v\rangle)^2+(\sin^2(z(\theta))\langle \hat\varphi_m(\theta),v\rangle^2
\right)^{1/2}d\theta.\label{below}
\end{align}
Write $F:=F^{\prime}\cup F^{\prime\prime}$, where
\begin{align*}
F^{\prime}&:=F\cap\left\{\theta\colon  |\langle \varphi_m(\theta),v\rangle -\cos(z(\theta))\langle \hat\varphi_m(\theta),v\rangle\big|\geq \frac{\langle \varphi_m(\theta),v\rangle}{2}\right\}\quad\text{and}\\
F^{\prime\prime}&:=F\cap\left\{\theta\colon  |\langle \varphi_m(\theta),v\rangle -\cos(z(\theta))\langle \hat\varphi_m(\theta),v\rangle\big|< \frac{\langle \varphi_m(\theta),v\rangle}{2}\right\}.
\end{align*}
On $F^{\prime}$, the integrand \eqref{below} is greater than or equal to $\langle \varphi_m(\theta),v\rangle/2$. On $F^{\prime\prime}$ the inequality $\cos(z(\theta))\langle \hat\varphi_m(\theta),v\rangle>\langle \varphi_m(\theta),v\rangle/2$ holds, and consequently
\begin{align*}
\langle\hat\varphi_m(\theta),v\rangle|\sin(z(\theta))|&>\frac{\langle \varphi_m(\theta),v\rangle}{2}|\sin(z(\theta))|\\
&>\frac{\langle \varphi_m(\theta),v\rangle}{\pi}|z(\theta)|\geq \frac{\langle \varphi_m(\theta),v\rangle}{\pi}\varepsilon^{\prime}.
\end{align*}
Altogether, this yields that the integrand in \eqref{below} is  larger than or equal to $\langle\varphi_m(\theta),v\rangle\varepsilon^{\prime}/\pi$. Now, it is easy to see that, due to Assumption~B.3, \eqref{Q1cons} cannot hold. This leads to a contradiction. 

Thus, we can conclude that $\max_{j\in\mathbb{Z}}\|\phi_{mj} - \hat \phi_{mj} \|=o_P(1)$, so that, for  sufficiently slowly growing $L$, we also have $L\,\max_{j\in\mathbb{Z}}\|\phi_{mj} - \hat \phi_{mj} \| = o_P(1)$. Consequently,
\begin{align}\label{eq:product:conv}
\Bigg|\Sum_{|j| \leq L} \ip{X_{k-j}}{\phi_{mj} - \hat \phi_{mj}}  \Bigg| &= o_P(1)\times \left( L^{-1}\Sum_{j=-L}^L \| X_{k-j} \|\right).
\end{align}
It remains to show that $L^{-1}\Sum_{j=-L}^L \| X_{k-j} \|=O_P(1)$. By the weak stationarity assumption, we have $E\|X_k\|^2=E\|X_1\|^2$, and hence, for any $x > 0$,
\begin{align*}
P\bigg(L^{-1}\Sum_{j=-L}^L \| X_{k-j} \| > x\bigg) \leq \frac{\sum_{k=-L}^LE \|X_k\|}{Lx} \leq \frac{3\sqrt{E \|X_1\|^2}}{x}.
\end{align*}

\vspace{-12mm}

\end{proof}

\vspace{2mm}

\begin{lemma}\label{lem:sum1}
Let $L=L(n)\to\infty$. Then, under condition \eqref{e:abssymcov}, we have
\begin{align*}
\Bigg|\Sum_{|j|>L} \ip{X_{k-j}}{\phi_{mj}}\Bigg|=o_P(1).
\end{align*}
\end{lemma}
\begin{proof}
This is immediate from Proposition~\ref{pr:secondorder}, part (a).
\end{proof}

Turning to  the proof of Proposition~\ref{lem:gammas}, we first establish the following lemma, which an extension to lag-$h$ autocovariance operators of a consistency result from \cite{hormann:kokoszka:2010} on the empirical covariance operator. Define, for $|h|<n$,
$$
\hat C_h=\frac{1}{n}\sum_{k=1}^{n-h}X _{k+h}\otimes X _{k},\quad h\geq 0,
\quad\text{and}\quad
\hat C_h=\hat C_{-h},\quad h< 0.
$$
\begin{lemma}\label{lem:be}
Assume that $(X_t : t \in \mathbb{Z})$ is an $L^4$-$m$-approximable series. 
Then, for all~$|h|<n$,  
$E \| \hat C_{h} - C_{h} \|_\cS \leq U\sqrt{{(|h| \vee 1)}/{n}},$
where the constant~$U$   neither depends on $n$ nor on $h$.
\end{lemma}

\begin{proof}[Proof of Proposition \ref{lem:gammas}]
By the triangle inequality,
\begin{align}
&2\pi \| \mathcal{F}_\theta^X - \mathcal{\hat F}_\theta^X \|_\cS= \Bigg\| \Sum_{k \in \mathbb{Z}} C_h e^{-\ii h\theta} - \Sum_{h=-q}^{q} \bigg( 1 - \frac{|h|}{q} \bigg)\hat C_h e^{-\ii h\theta} \Bigg\|_\cS\nonumber\\
&\qquad\leq \Bigg\| \Sum_{h = -q}^q \bigg(1 - \frac{|h|}{q}\bigg) (C_h - \hat C_h) e^{-\ii h\theta} \Bigg\|_\cS\nonumber\\
&\qquad\ \ \ + 
\Bigg\| \frac{1}{q} \Sum_{h = -q}^q |h|C_h e^{-\ii h\theta} \Bigg\|_\cS +
\Bigg\| \Sum_{|h| > q} C_h e^{-\ii h\theta} \Bigg\|_\cS\nonumber\\
&\qquad\leq  \Sum_{h = -q}^q \bigg(1 - \frac{|h|}{q}\bigg) \|C_h - \hat C_h\|_\cS + 
\frac{1}{q} \Sum_{h = -q}^q |h| \| C_h \|_\cS +
\Sum_{|h| > q} \| C_h \|_\cS.\nonumber
\end{align}
The last two terms tend to $0$ by condition \eqref{e:abssymcov} and Kronecker's lemma. For the first term we may use Lemma~\ref{lem:be}. Taking   expectations, we obtain that,  for some $U_1$, \vspace{-1mm}
\begin{align*}
\Sum_{h = -q}^q \bigg(1 - \frac{|h|}{q}\bigg) E \|C_h - \hat C_h\|_\cS \leq U_1\frac{q^{3/2}}{\sqrt{n}}.
\vspace{-1mm}\end{align*} 
Note that the bound does not depend on $\theta$; hence $q^3=o(n)$ and condition \eqref{e:abssymcov} jointly imply that $\sup_{\theta\in[-\pi,\pi]}E\| \mathcal{F}_\theta^X - \mathcal{\hat F}_\theta^X \|_\cS\to 0$ as $n\to\infty$.
\end{proof}

\section{Technical results and background}\label{app:technicalresults}

\subsection{Linear operators}\label{se:lo}
Consider the class $\mathcal{L}(H,H^{\prime})$ of bounded linear operators between two Hilbert spaces $H$ and $H^{\prime}$.  For~$\Psi\in\mathcal{L}(H,H^{\prime})$, the {\it operator norm} is defined as 
$\|\Psi\| _\mathcal{L}:=\sup_{\|x\|\leq 1}\|\Psi(x)\|$. The simplest operators can be defined via a tensor product $v\otimes w$; then  $v\otimes w(z):=v\langle z,w\rangle$. 
Every operator $\Psi\in\mathcal{L}(H,H^{\prime})$ possesses an {\it adjoint} $\Psi^*\in\mathcal{L}(H^{\prime},H)$, which  satisfies $\langle \Psi (x),y\rangle=\langle x,\Psi^*(y)\rangle$ for all~$x\in H$ and $y\in H^{\prime}$. It holds that $\|\Psi^*\|_\mathcal{L}=\|\Psi\|_\mathcal{L}$. If $H=H^{\prime}$, then $\Psi$ is called {\em self-adjoint} if $\Psi=\Psi^*$. It is called non-negative definite  if $\langle \Psi x,x\rangle\geq 0$ for all $x\in H$.

A linear operator $\Psi\in\mathcal{L}(H,H^{\prime})$ is said to be {\em Hilbert-Schmidt} if, for some orthonormal basis $(v_k\colon k\geq 1)$ of $H$, we have $\|\Psi\|_\mathcal{S}^2:=\sum_{k\geq 1}\|\Psi(v_k)\|^2<\infty$. Then,~$\|\Psi\|_\mathcal{S}$ defines a norm, the so-called {\it Hilbert-Schmidt norm} of $\Psi$, which bounds the operator norm $\|\Psi\|_\mathcal{L}\leq \|\Psi\|_\mathcal{S}$, and can be shown to be independent of the choice of the orthonormal basis. Every Hilbert-Schmidt operator is compact. The class of Hilbert-Schmidt operators between $H$ and $H^{\prime}$ defines again a separable Hilbert space with inner product $\langle \Psi,\Theta\rangle_\mathcal{S}:=\sum_{k\geq 1}\langle \Psi(v_k),\Theta(v_k)\rangle$: denote this class by~$\mathcal{S}(H,H^{\prime})$.\vspace{1mm}

If $\Psi\in\mathcal{L}(H,H^{\prime})$ and $\Upsilon\in\mathcal{L}(H^{\prime\prime},H)$, then $\Psi\Upsilon$ is the operator   mapping  $x\in H^{\prime\prime}$ to~$\Psi(\Upsilon(x))\in H^{\prime}$.
Assume that $\Psi$ is a compact operator in $\mathcal{L}(H,H^{\prime})$ and let $(s^2_j)$ be the eigenvalues of $(\Psi^*)\Psi$. Then $\Psi$ is said to be {\em trace class} if $\|\Psi\|_\mathcal{T}:=\sum_{j\geq 1} s_j<\infty$. In this case, $\|\Psi\|_\mathcal{T}$ defines a norm, the so-called {\em Schatten 1-norm}. We have that \linebreak $\|\Psi\|_\mathcal{S}\leq\|\Psi\|_\mathcal{T}$, and hence any trace-class operator is Hilbert-Schmidt. For self-adjoint non-negative operators, it holds that $\|\Psi\|_\mathcal{T}=\tr(\Psi):=\sum_{k\geq 1}\langle \Psi(v_k),v_k\rangle$. If~$\tilde\Psi\tilde\Psi=\Psi$, then we have $\tr(\Psi)=\|\tilde\Psi\|_\mathcal{S}^2$.

For further background on the theory of linear operators we refer to \cite{gohberg:goldberg:kaashoek:2003}.

\subsection{Random sequences in Hilbert spaces}
All random elements that appear in the sequel are assumed to be defined on a common probability space~$(\Omega,\mathcal{A},P)$. We write $X\in L_H^p(\Omega,\mathcal{A},P)$ (in short, $X\in L_H^p$) if $X$ is an $H$-valued random variable such that~$E\|X\|^p<\infty$. Every element $X\in L_H^1$ possesses an expectation, which is the unique $\mu\in H$ satisfying $E\langle X,y\rangle=\langle\mu,y\rangle$ for all $y\in H$.
Provided that $X$ and $Y$ are in~$L_H^2$, we can define the cross-covariance operator as $C_{XY}:=E(X-\mu_X)\otimes (Y-\mu_Y)$, where $\mu_X$ and $\mu_Y$ are the expectations of $X$ and $Y$, respectively. We have that $\|C_{XY}\|_\mathcal{T}\leq E\|(X-\mu_X)\otimes (Y-\mu_Y)\|_\mathcal{T}=E\|X-\mu_X\|\|Y-\mu_Y\|$, and so these operators are trace-class. An important specific role is played by the covariance operator $C_{XX}$. This operator is non-negative definite and self-adjoint  with $\tr(C_{XX})=E\|X-\mu_X\|^2$. An $H$-valued process $(X_t)$ is called  {\em (weakly) stationary}  if $(X_t)\in L_H^2$, and  $EX_t$ and~$C_{X_{t+h}X_t}$ do not depend on~$t$. In this case, we write $C_h^X$, or shortly $C_h$, for~$C_{X_{t+h}X_t}$ if it is clear to which process it belongs. 

Many useful results on random processes in Hilbert spaces or more general Banach spaces are collected in Chapters 1 and 2 of \cite{bosq:2000}.

\subsection{Proofs for Appendix~A}

\begin{proof}[Proof of Proposition~\ref{pr:Hfourier}]
Letting $0<m<n$, note that \vspace{-2mm}
\begin{align*}
\|S_n-S_m\|_2^2&=\bigg(\sum_{m\leq |k|\leq n} f_k e_k,\sum_{m\leq |\ell|\leq n} f_\ell e_\ell\bigg)\\
&=\frac{1}{2\pi}\int_{-\pi}^\pi \sum_{m\leq |k|\leq n}\sum_{m\leq |\ell|\leq n}\langle f_k,f_\ell\rangle e^{\ii (k-\ell)\theta}d\theta =\sum_{m\leq |k|\leq n}\|f_k\|^2.
\end{align*}
To prove the first statement, we need to show that $(S_n)$ defines a Cauchy sequence in~$L_H^2([-\pi,\pi])$, which follows if we show that $\sum_{k\in\mathbb{Z}}\|f_k\|^2<\infty$. 
We use the fact that, for any $v\in H$, the function $\langle x(\theta),v\rangle$ belongs to $L^2([-\pi,\pi])$. Then, by Parseval's identity and \eqref{e:defint}, we have, for any $v \in H$, \vspace{-1mm}
\begin{align*}
\frac{1}{2\pi}\int_{-\pi}^\pi|\langle x(\theta),v\rangle|^2d\theta&=\sum_{k\in\mathbb{Z}}
\left(\frac{1}{2\pi}\int_{-\pi}^\pi\langle x(s),v\rangle e^{-\ii ks}ds\right)^2 =\sum_{k\in\mathbb{Z}}|\langle f_k,v\rangle|^2.
\end{align*}
Let $(v_k\colon k\geq 1)$ be an orthonormal basis of $H$. Then, by the last result and Parseval's identity again, it follows that 
\begin{align*}
\|x\|_2^2&=\frac{1}{2\pi}\int_{-\pi}^\pi\sum_{\ell\geq 1}|\langle x(\theta),v_\ell\rangle|^2d\theta =\frac{1}{2\pi}\sum_{\ell\geq 1}\int_{-\pi}^\pi|\langle x(\theta),v_\ell\rangle|^2d\theta\\
&=\sum_{\ell\geq 1}\sum_{k\in\mathbb{Z}}|\langle f_k,v_\ell\rangle|^2=\sum_{k\in\mathbb{Z}}\|f_k\|^2.
\end{align*}

As for the second statement, we conclude from classical Fourier analysis results that, for each $v\in H$,\vspace{-1mm}
$$
\lim_{n\to\infty}\frac{1}{2\pi}\int_{-\pi}^\pi\left(\langle x(\theta),v\rangle-\sum_{k=-n}^n\left(\frac{1}{2\pi}\int_{-\pi}^\pi\langle x(s),v\rangle e^{-\ii ks}ds\right) e^{\ii k\theta}\right)^2d\theta=0.
\vspace{-1mm}$$
Now, by definition of $S_n$, this is equivalent to\vspace{-1mm}
$$
\lim_{n\to\infty}\frac{1}{2\pi}\int_{-\pi}^\pi \left\langle x(\theta)-S_n(\theta),v\right\rangle^2d\theta=0,\quad \forall v\in H.
\vspace{-1mm}$$
Combined with the first statement of the proposition and \vspace{-1mm}
\begin{align*}
\int_{-\pi}^\pi \left\langle x(\theta)-S(\theta),v\right\rangle^2d\theta&\leq 2\int_{-\pi}^\pi \left\langle x(\theta)-S_n(\theta),v\right\rangle^2d\theta\\ &\quad+ 2\|v\|^2\int_{-\pi}^\pi \| S_n(\theta)-S(\theta)\|^2d\theta,
\end{align*}
this implies that
\begin{equation}\label{eq:fourier_lim}
\frac{1}{2\pi}\int_{-\pi}^\pi \left\langle x(\theta)-S(\theta),v\right\rangle^2d\theta=0,\quad \forall v\in H.
\end{equation}
Let $(v_i)$, $i\in\mathbb{N}$ bee an orthonormal basis of $H$, and define \vspace{-2mm}
 $$A_i :=\{\theta\in[-\pi,\pi]\colon \left\langle x(\theta)-S(\theta),v_i\right\rangle\neq 0\}.\vspace{-2mm}$$
 By~\eqref{eq:fourier_lim}, we have that $\lambda(A_i)=0$ ($\lambda$ denotes the Lebesgue measure), and hence $\lambda(A)=~0$ for $A=\cup_{i\geq 1}A_i$. Consequently, since $(v_i)$ define an orthonormal basis, for any $\theta\in [-\pi,\pi]\setminus A$, we have $\langle x(\theta)-S(\theta),v\rangle=0$ for all $v\in H$, which in turn implies that $x(\theta)-S(\theta)=0$.
\end{proof}

\begin{proof}[Proof of Proposition~\ref{pr:summab}]
Without loss of generality, we assume that $EX_0=0$.
Since~$X_0$ and $X_h^{(h)}$, $h\geq 1$,  are independent,  
$$
\|C_h^X\|_\mathcal{S}=\|EX_0\otimes(X_h-X_h^{(h)})\|_\mathcal{S}\leq (E\|X_0\|^2)^{1/2}(E\|X_h-X_h^{(h)}\|^2)^{1/2}.
$$
The first statement of the proposition follows.

%
Let  $\theta$ be fixed. Since $\mathcal{F}_\theta^X$ is non-negative and self-adjoint, it is trace class if and only if
\begin{equation}\label{eq:trace_class}
\tr(\mathcal{F}_\theta^X)=\sum_{m\geq 1}\langle \mathcal{F}_\theta^X(v_m),v_m\rangle<\infty
\end{equation} 
for some orthonormal basis $(v_m)$ of $H$.  The trace can be shown to be independent of the choice of the basis.
Define $V_{n,\theta}=(2\pi n)^{-1/2}\sum_{k=1}^n X_k e^{\ii k\theta}$ and note that, by stationarity, 
$$
\mathcal{F}_{n,\theta}^X:=EV_{n,\theta}\otimes V_{n,\theta}=\frac{1}{2\pi}\sum_{|h|< n}\left(1-\frac{|h|}{n}\right) EX_0\otimes X_{-h}e^{-\ii h\theta}.
$$
It is easily verified that the operators $\mathcal{F}_{n,\theta}^X$ again are non-negative and self-adjoint. Also note that, by the triangular inequality,
$$
\|\mathcal{F}_{n,\theta}^X-\mathcal{F}_{\theta}^X\|_\mathcal{S}\leq\sum_{|h|<n}\frac{|h|}{n}\|C_h^X\|_\mathcal{S}+\sum_{|h|\geq n}\|C_h^X\|_\mathcal{S}.
$$
By application of \eqref{e:abssymcov} and Kronecker's lemma, it easily follows that the latter two terms converge to zero. This implies that $\mathcal{F}_{n,\theta}^X(v)$ converges in norm to~$\mathcal{F}_{\theta}^X(v)$, for any $v\in H$.

Choose $v_m=\varphi_m(\theta)$. Then, by continuity of the inner product and Fatou's lemma, we have
\begin{align*}
\sum_{m\geq 1}\langle \mathcal{F}_\theta^X(\varphi_m(\theta)),\varphi_m(\theta)\rangle&=\sum_{m\geq 1}\liminf_{n\to\infty}\langle \mathcal{F}_{n,\theta}^X(\varphi_m(\theta)),\varphi_m(\theta)\rangle\\&\leq\liminf_{n\to\infty}\sum_{m\geq 1}\langle \mathcal{F}_{n,\theta}^X(\varphi_m(\theta)),\varphi_m(\theta)\rangle.
\end{align*}
Using the fact that the $\mathcal{F}_{n,\theta}^X$'s are self-adjoint and non-negative, we get
\begin{align*}
\sum_{m\geq 1}\langle \mathcal{F}_{n,\theta}^X(\varphi_m(\theta)),\varphi_m(\theta)\rangle&=\tr(\mathcal{F}_{n,\theta}^X)=E\|V_n\|^2\\
&=\frac{1}{2\pi}\sum_{|h|< n}\left(1-\frac{|h|}{n}\right) E\langle X_0,X_{h}\rangle e^{-\ii h\theta}.
\end{align*}
Since $|E\langle X_0,X_h\rangle|=|E\langle X_0,X_h-X_h^{(h)}\rangle|$, by the Cauchy-Schwarz inequality, $$\sum_{h\in\mathbb{Z}}|E\langle X_0,X_h\rangle|\leq \sum_{h\in\mathbb{Z}}(E\| X_0\|^2)^{1/2}(E(X_h-X_h^{(h)})^2)^{1/2}<\infty,$$ and thus the dominated convergence theorem implies that
\begin{equation}\label{trace}
\tr(\mathcal{F}_\theta^X)\leq \sum_{h\in\mathbb{Z}}|E\langle X_0,X_h\rangle|<\infty,
\end{equation}
which completes he proof.\end{proof}

\begin{proof}[Proof of Proposition~\ref{pr:conteigen}]
We have (see e.g.\ \cite{gohberg:goldberg:kaashoek:2003}, p.\ 186) that the dynamic eigenvalues are such that 
$
|\lambda_m(\theta)-\lambda_m(\theta')|\leq\|\mathcal{F}_\theta^X-\mathcal{F}_{\theta'}^X\|_\mathcal{S}.
$
Now,
$$
\|\mathcal{F}_\theta^X-\mathcal{F}_{\theta'}^X\|_\mathcal{S}\leq\sum_{h\in\mathbb{Z}}\|C_h^X\|_\mathcal{S}|e^{-\ii h\theta}-e^{-\ii h\theta'}|.
$$
The summability condition \eqref{e:abssymcov} implies continuity, hence part (a) of the proposition. The fact that $|e^{-\ii h\theta}-e^{-\ii h\theta'}|\leq |h||\theta-\theta'|$ yields part (b).  To prove (c),  observe that
$$
\lambda_m(\theta)\varphi_m(\theta)=\mathcal{F}_\theta^X(\varphi_m(\theta))=\frac{1}{2\pi}\sum_{h\in\mathbb{Z}}EX_h\langle \varphi_m(\theta),X_0\rangle e^{-\ii h\theta}
$$
 for any $\theta\in[-\pi,\pi]$. Since the eigenvalues $\lambda_m(\theta)$ are real, we obtain, by computing the complex conjugate of the above equalities,
$$
\lambda_m(\theta) \overline{\varphi_m(\theta)}=\frac{1}{2\pi}\sum_{h\in\mathbb{Z}}EX_h\langle \overline{\varphi_m(\theta)},X_0\rangle e^{\ii h\theta}=\mathcal{F}_{-\theta}^X(\overline{\varphi_m(\theta)}).
$$
This shows that $\lambda_m(\theta)$ and $\overline{\varphi_m(\theta)}$ are eigenvalue and eigenfunction of $\mathcal{F}^X_{-\theta}$ and they must correspond to a pair $(\lambda_n(-\theta),\varphi_n(-\theta))$; (c) follows.
\end{proof}

\begin{lemma}\label{le:trace}
Let $(Z_t)$ be a stationary sequence in $L_H^2$ with spectral density $\mathcal{F}^Z_\theta$. Then, 
$$\int_{-\pi}^\pi \tr(\mathcal{F}_\theta^Z)d\theta=\tr\Big(\int_{-\pi}^\pi \mathcal{F}_\theta^Zd\theta\Big)=\tr(C_0^Z)=E\|Z_t\|^2.$$ 
\end{lemma}
\begin{proof} Let $\mathcal{S}=\mathcal{S}(H,H)$. Note that $\int_{-\pi}^\pi \mathcal{F}_\theta^Zd\theta=\mathrm{I}\mathcal{F}^Z$ if and only if
\begin{equation}\label{e:IGamma}
\langle \mathrm{I}\mathcal{F}^Z,V\rangle_\mathcal{S}=\int_{-\pi}^\pi\langle  \mathcal{F}_\theta^Z,V\rangle_\mathcal{S}\,d\theta\quad\text{for all }V\in\mathcal{S}.
\end{equation}
For some orthonormal basis~$(v_k)$ define $V_N=\sum_{k=1}^Nv_k\otimes v_k$. Then~\eqref{e:IGamma} implies that 
\begin{align*}
\tr(\mathrm{I}\mathcal{F}^Z)&=\lim_{N\to\infty}\sum_{k=1}^N\langle \mathrm{I}\mathcal{F}^Z(v_k),v_k\rangle     =\lim_{N\to\infty}\langle\mathrm{I}\mathcal{F}^Z,V_N\rangle_\mathcal{S}       \\
&=\lim_{N\to\infty}\int_{-\pi}^ \pi\langle\mathcal{F}_\theta^Z,V_N\rangle_\mathcal{S}d\theta
=\lim_{N\to\infty}\int_{-\pi}^\pi\sum_{k=1}^N\langle  \mathcal{F}_\theta^Z(v_k),v_k\rangle\,d\theta.
\end{align*}
Since $\mathcal{F}_\theta^Z$ is non-negative definite for any $\theta$, the monotone convergence theorem allows to interchange the limit with the integral.
\end{proof}

\begin{proof}[Proof of Proposition~\ref{pr:sd_filter0}]
(i)  Define $Y_t^{r,s}:=\sum_{r<|k|\leq s}\Psi_k(X_{t-k})$ and the related transfer operator $\Psi^{r,s}_\theta :=\sum_{r<|k|\leq s}\Psi_k e^{-\ii k\theta}$. We also use $Y_t^s=\Psi_0(X_t)+Y_t^{0,s}$ and $\Psi^{s}_\theta=\Psi_0+\Psi^{0,s}_\theta$. Since~$Y_t^{r,s}$ is a finite sum, it is obviously  in $L^2_{H^{\prime}}$.  Also, the finite number of filter coefficients makes it easy to check that $(Y_t^{r,s}\colon t\in\mathbb{Z})$ is stationary and has spectral density operator $\mathcal{F}_\theta^{Y^{r,s}}=\Psi_\theta^{r,s}\mathcal{F}_\theta^X(\Psi_\theta^{r,s})^*$.  By the previous  lemma we have
\begin{align*}
E\|Y_t^{r,s}\|^2&=\int_{-\pi}^\pi \tr(\mathcal{F}_\theta^{Y_t^{r,s}})d\theta=
\int_{-\pi}^\pi\tr(\Psi_\theta^{r,s}\mathcal{F}_\theta^X(\Psi_\theta^{r,s})^*)d\theta\\
&\leq\int_{-\pi}^\pi\|\Psi_\theta^{r,s}\|^2_{\mathcal{S}(H,H^{\prime})}\tr(\mathcal{F}_\theta^X)d\theta.
\end{align*}
Now, it  directly follows from the  assumptions that $(Y_t^s\colon s\geq 1)$ defines a Cauchy sequence in $L^2_{H^{\prime}}$. This proves (i).


Next, remark that by our assumptions $\Psi_\theta\mathcal{F}_\theta^X(\Psi_\theta)^* \in L^2_{\mathcal{S}(H^{\prime},H^{\prime})}([-\pi,\pi])$. Hence, by the results in Appendix~A.1, 
$$
\sum_{|h|\leq r}\frac{1}{2\pi}\int_{-\pi}^\pi\Psi_u\mathcal{F}_u^X(\Psi_u)^*e^{\ii hu}du\, e^{-\ii h\theta}\to \Psi_\theta\mathcal{F}_\theta^X(\Psi_\theta)^*\quad\text{as } r\to\infty ,
$$
where convergence is in $L^2_{\mathcal{S}(H^{\prime},H^{\prime})}([-\pi,\pi])$.
We prove that $\Psi_\theta\mathcal{F}_\theta^X(\Psi_\theta)^*$ is the spectral density operator of $(Y_t)$. This is the case if $\frac{1}{2\pi}\int_{-\pi}^\pi\Psi_u\mathcal{F}_u^X(\Psi_u)^*e^{\ii hu}du=C_h^Y$. 
 For the approximating sequences $(Y_t^s\colon t\in\mathbb{Z})$ we know from (i) and Remark~\ref{rem3} that
 $$
\frac{1}{2\pi}\int_{-\pi}^\pi\Psi_u^s\mathcal{F}_u^X(\Psi_u^s)^*e^{\ii hu}du=C_h^{Y^s}.
$$
Routine arguments show that under our assumptions $\|C_h^{Y^s}-C_h^Y\|_{\mathcal{S}(H',H')}\to 0$ and
 $$
\left\|\int_{-\pi}^\pi\left(\Psi_u\mathcal{F}_u^X(\Psi_u)^*-
\Psi_u^s\mathcal{F}_u^X(\Psi_u^s)^*\right)e^{\ii hu}du\right\|_{\mathcal{S}(H',H')}\to 0,\quad(s\to\infty).
$$
Part (ii) of the proposition  follows, hence also part~(iii). \end{proof}

\begin{proof}[Proof of Lemma~\ref{lem:be}] Let us only consider the case $h\geq0$. Define $X_n^{(r)}$ as the $r$-dependent approximation of $(X_n)$ provided by Definition~\ref{def:mapr}. Observe that
\begin{equation*}
nE\big\|\hat{C}_{h} - C_{h}\big\|_\cS^2 = nE\left\|\frac{1}{n}\sum\limits_{k=1}^{n-h}Z_k\right\|_\cS^2,
\end{equation*}
where $Z_k = X_{k+h} \otimes X_k - C_{h}$. Set $Z_k^{(r)} = X_{k+h}^{(r)} \otimes X_k^{(r)} - C_{h}$. Stationarity of $(Z_k)$~implies 
\begin{align}
nE \left\|\frac{1}{n}\sum\limits_{k=1}^{n-h} Z_k\right\|_\cS^2 &= \sum\limits_{|r|<n-h}\left(1- \frac{|r|}{n}\right)E\ip{Z_0}{Z_r}_\mathcal{S}\nonumber\\
&\leq \sum\limits_{r=-h}^{h} |E\ip{Z_0}{Z_r}_\mathcal{S}| + 2\sum\limits_{r=h+1}^\infty |E\ip{Z_0}{Z_r}_\mathcal{S}|,\label{sum}
\end{align}
while  the Cauchy-Schwarz inequality yields
 $$|E\ip{Z_0}{Z_r}_\mathcal{S}| \leq E|\ip{Z_0}{Z_r}_\mathcal{S}| \leq \sqrt{E{\|Z_0\|_\mathcal{S}^2}E{\|Z_r\|_\mathcal{S}^2}} = E{\|Z_0\|_\mathcal{S}^2}.$$ Furthermore, from $\|X_h\otimes X_0\|= \|X_h\| \|X_0\|$, we deduce
\begin{align*}
E\|Z_0\|_\cS^{2} =  E\|X_0\|^{2} \|X_{h}\|^{2} \leq \left(E \|X_0\|^4\right)^{1/2} <\infty.
\end{align*}
Consequently, we can bound the first sum in \eqref{sum} by $(2h+1)\left(E\|X_0\|^4\right)^{1/2}$. For the   second term in \eqref{sum}, we obtain, by independence of~$Z_r^{(r-h)}$ and $Z_0$, that
\begin{align*}
|E\ip{Z_0}{Z_r}_\cS|=|E\ip{Z_0}{Z_r-Z_r^{(r-h)}}_\cS| \leq (E\|Z_0\|_\cS^2)^{1/2} (E\|Z_r - Z_r^{(r-h)}\|_\cS^2)^{1/2}.
\end{align*}
To conclude, it suffices to show that $\sum _{r=1}^\infty (E\|Z_r - Z_r^{(r-h)}\|_\cS^2)^{1/2}\leq M<\infty$, where the bound $M$ is independent of $h$. Using an inequality of the type $|ab-cd|^{2} \leq 2|a|^{2}|b-d|^{2} + 2|d|^{2}|a-c|^{2}$, we obtain
\begin{align*}
&E\|Z_r - Z_r^{(r-h)}\|_\cS^2= E\|X_r\otimes X_{r+h} - X_r^{(r-h)} \otimes X_{r+h}^{(r-h)}\|_\cS^2\\
&\qquad\leq 2E \|X_r\|^2 \|X_{r+h} - X_{r+h}^{(r-h)}\|^2 + 2E\|X_{r+h}^{(r-h)}\|^2 \|X_r - X_r^{(r-h)}\|^2\\
&\qquad\leq 2 (E \|X_r\|^4)^{1/2} (E\|X_{r+h} - X_{r+h}^{(r-h)}\|^4)^{1/2}\\
&\qquad\quad+ 2 (E\|X_{r+h}^{(r-h)}\|^4)^{1/2} (E\|X_r - X_r^{(r-h)}\|^4)^{1/2}.
\end{align*}
Note that $E\|X_{r}\|^4=E\|X_{r+h}^{(r-h)}\|^4=E\|X_0\|^4$ and  
$$E\|X_{r+h} - X_{r+h}^{(r-h)}\|^4=E\|X_{r} - X_{r}^{(r-h)}\|^4=E\|X_{0} - X_{0}^{(r-h)}\|^4.$$
 Altogether we get
$$
E\|Z_r - Z_r^{(r-h)}\|_\cS^2\leq 4  (E \|X_0\|^4)^{1/2} (E\|X_{0} - X_{0}^{(r-h)}\|^4)^{1/2}.
$$
Hence, $L^4$-$m$-approximability implies that $\sum_{r=h+1}^\infty |E\ip{Z_0}{Z_r}_\mathcal{S}|$ converges and is uniformly bounded over $0\leq h<n$. 
\end{proof}

\section*{Acknowledgement} The research of Siegfried H\"ormann and {\L}ukasz Kidzi\'nski was supported by the Communaut\'e fran\c{c}aise de Belgique -- Actions de Recherche Concert\'ees (2010--2015) and the Belgian Science Policy Office -- Interuniversity attraction poles (2012--2017). The research of Marc Hallin was supported by the Sonderforschungsbereich ``Statistical modeling of nonlinear dynamic processes'' (SFB823) of the Deutsche Forschungsgemeinschaft and the Belgian Science Policy Office -- Interuniversity attraction poles (2012--2017).

\end{document}